\documentclass[amstex,12pt,reqno]{amsart}

\usepackage{amsmath,amsfonts,amssymb,amsthm,enumerate,hyperref,multicol}
\usepackage{mathrsfs}
\usepackage{graphicx}
\usepackage{epstopdf}
\usepackage{epsfig}

\newtheorem{lemma}{Lemma}
\newtheorem{thm}{Theorem}
\newtheorem{prop}{Proposition}
\newtheorem{definition}{Definition}

\newtheorem{corollary}{Corollary}
\newtheorem{remark}{Remark}

\numberwithin{equation}{section}

\title[Approximation by Neural Network Sampling Operators in Mixed Lebesgue Spaces]
{Approximation by Neural Network Sampling Operators in Mixed Lebesgue Spaces}



\author{Arpan Kumar Dey}

\address{Department of Mathematics, Indian Institute of Technology Madras, Chennai-600036, Tamil Nadu, India}

\email{arpanmathematics@gmail.com }

\author{A. Sathish Kumar}

\address{Department of Mathematics, Indian Institute of Technology Madras, Chennai-600036, Tamil Nadu, India}

\email{sathishkumar@iitm.ac.in, mathsatish9@gmail.com}

\author{P. Devaraj}

\address{School of Mathematics, Indian Institute of Science Education and Research, Thiruvananthapuram, India}

\email{devarajp@iisertvm.ac.in}

\begin{document}

\begin{abstract}
In this paper, we prove the rate of approximation for the Neural Network Sampling Operators activated by sigmoidal functions with mixed Lebesgue norm in terms of averaged modulus of smoothness for a bounded measurable functions on bounded domain. 
In order to achieve the above result, we first establish that the averaged modulus of smoothness is finite for certain suitable subspaces of $L^{p,q}(\mathbb{R}\times\mathbb{R}).$
Using the properties of averaged modulus of smoothness, we estimate the rate of approximation of certain linear operators in mixed Lebesgue norm. Then, as an application of these linear operators, we obtain the Jackson type approximation theorem, in order to give a characterization for the rate of approximation of neural network operators in-terms of averaged modulus of smoothness in mixed norm. Lastly, we discuss some examples of sigmoidal functions and using these sigmoidal functions, we show the implementation of continuous and discontinuous functions by neural network operators .

\end{abstract}

\maketitle
\section{Introduction and Preliminaries}\label{section1}

In sampling and reconstruction, the signals are usually defined in time domains or space fields. However, in practice, some signals are time-varying. This implies that the signals can live in time-space domains at the same time. A suitable tool for analyzing such time-space signals is mixed Lebesgue spaces. Because in mixed Lebesgue spaces, we can consider the integrability of each variable independently. This is major different from the classical Lebesgue spaces, see \cite{Aben1, Aben2, Fer} and the references therein.
Benedek was the first person who initiated the study of mixed Lebesgue spaces, see \cite{Aben1, Aben2}. Inspired by Benedek works, Fernandez \cite{Fer} analysed the singular integral operators for functions in mixed Lebesgue spaces. Further, Francia, Ruiz and Torrea \cite{Fran} developed the theory for Calder\'on-Zygmund operators in the settings of mixed Lebesgue spaces. Recently, several researchers have analyzed the sampling and reconstruction problems in mixed Lebesgue spaces, see \cite{Rto, liu, jia, siva1, deva1, jia12} and the references therein. In this paper, we analyze the Neural Network operators in mixed Lebesgue spaces in view of approximation theory.

The feed-forward neural network operators (FNNs) with one hidden layer are defined by
\begin{eqnarray*}
N_{p}(x):=\sum_{j=0}^{p}c_{j}\sigma(\langle \alpha_{j}. x\rangle+\beta_{j}), \,\,\,\ x\in\mathbb{R}^{n},\,\,\,\ n \in\mathbb{N}
\end{eqnarray*}
where $\beta_{j}\in\mathbb{R},$ $0\leq j\leq p,$ are called thresholds, $\alpha_{j}\in\mathbb{R}^{n}$ are the connection weights, $c_{j}\in\mathbb{R}$ are the coefficients. Furthermore, $\langle \alpha_{j} . x\rangle$ denotes the inner product of $\alpha_j$ and $x$, and $\sigma$ denotes activation function of the neural network.

In past three decades, the approximation by NN operators have been widely analyzed by several researchers in various settings. Using the tools from Functional Analysis, Cybenko analysed the approximation of continuous function by NN operators in \cite{Cybenko}. The approximation of continuous functions by NN operators with single and two hidden layers are discussed in \cite{Funahashi}. Hornik et al. \cite{Hornik} approximated the measurable functions by multilayer feedforward networks. Anastassiou \cite{Ana1} initiated the study of rate of approximation by NN operators in-terms of modulus of smoothness. Further, he discussed some quantitative convergence results in \cite{Ana2}. Recently, various approximation results for the neural network operators have been discussed in different settings, see \cite{GAM, DCR, cc1, cc2, cc3, cc4, cc5, cc6, cc7, cc8} and the references therein.\\
We begin with the definition of mixed Lebesgue spaces. For $1\leq p,q<\infty,$ let $L^{p,q}(\mathbb{R}\times \mathbb{R}^{d})$ denote the space of all complex valued measurable functions $f$ defined on $\mathbb{R}\times \mathbb{R}^{d}$ such that
$$\|f\|_{L^{p,q}}:=\left(\int_{\mathbb{R}}\left(\int_{\mathbb{R}^{d}}|f(x,\mathbf{y})|^{q}d\mathbf{y}\right )^{\frac{p}{q}}dx\right)^{\frac{1}{p}}<\infty.$$
Further, let $L^{\infty, \infty}(\mathbb{R} \times \mathbb{R}^{d})$ denote the set of all complex valued measurable functions on $\mathbb{R}^{d+1}$ such that $\|f\|_{L^{\infty, \infty}} := ess \sup |f| < \infty$.
For $ 1\leq p,q < \infty$, $\ell^{p,q}(\mathbb{Z}\times \mathbb{Z}^{d})$ denotes the space of all complex sequences $c=\big(c(k_{1},k_{2})\big)_{(k_{1}\in \mathbb{Z}, k_{2}\in \mathbb{Z}^{d})}$ such that
$$\|c\|_{\ell^{p,q}} := \left(\sum_{k_{1}\in \mathbb{Z}}\left(\sum_{k_{2} \in \mathbb{Z}^{d}}|c(k_{1}, k_{2})|^{q}\right)^{p/q}\right)^{1/p} < \infty.$$
We denote by $\ell^{\infty, \infty}(\mathbb{Z} \times \mathbb{Z}^{d}),$ the space of all complex sequences on $\mathbb{Z}^{d+1}$ such that $$\|c\|_{\ell^{\infty, \infty}} := \displaystyle \sup_{k \in \mathbb{Z}^{d+1}}|c(k)| < \infty.$$
We note that $L^{p,p}(\mathbb{R} \times \mathbb{R}^{d}) = L^{p}(\mathbb{R}^{d+1})$ and $\ell^{p,p}(\mathbb{Z} \times \mathbb{Z}^{d}) = \ell^{p}(\mathbb{Z}^{d+1})$ for $1 \leq p < \infty.$ \\

In this paper, we analyze neural network sampling operators in certain suitable subspaces of $L^{p,q}(\mathbb{R}\times\mathbb{R}).$ We observe that the sampling operators are not well defined in $L^{p,q}(\mathbb{R}\times\mathbb{R}),$, so we need to restrict the sampling operator to certain suitable subspace of $L^{p,q}(\mathbb{R}\times\mathbb{R}).$ Motivated by the work of Bardaro et al. \cite{Lpbut}, we define the subspace $\Lambda^{p,q}(\mathbb R\times \mathbb R^{d}),$ and analyze the sampling type neural network operators in $\Lambda^{p,q}(\mathbb R\times \mathbb R^{d}).$. In order to define these subspace, we need certain special type of partitions called admissible partitions, which is defined below: 
 
\begin{definition}
We say that a sequence $\Sigma_{x}=(x_{k})_{k\in \mathbb Z} \subset \mathbb R $ is an admissible partition of $\mathbb R$ if
$$\displaystyle 0< \underline{\Delta}:= \inf_{k\in \mathbb Z}(x_{k}-x_{k-1})\leq \sup_{k\in \mathbb Z}(x_{k}-x_{k-1}):= \bar{\Delta}<\infty,$$
where $\bar{\Delta}$ and $\underline{\Delta}$ denotes upper and lower mesh size respectively.
Again in $\mathbb R^d$ for every $i=1,2,...,d$ we take the admissible partition over the $i^{th}$ axis  $\Sigma_{i}=(y_{i,j_{i}})_{j_{i}\in \mathbb Z}$ such that
$$\displaystyle 0< \underline{\Delta}:=\min_{i=1,...,d} {\inf_{{j_{i}\in \mathbb Z}}}(y_{i,j_{i}}-y_{i,j_{i}-1})\leq \max_{i=1,...,d} {\sup_{{j_{i}\in \mathbb Z}}}(y_{i,j_{i}}-y_{i,j_{i}-1}):=\bar{\Delta}<\infty.$$

A sequence $\Sigma_{y}=(y_{j})_{j\in \mathbb Z^d}, y_{j}=(y_{1,j_{1}},...,y_{d,j_{d}})$ and $j=(j_{1},...,j_{d})$ is called an admissible sequence if it is written as
 $$\Sigma=\Sigma_{1}\times \Sigma_{2}\times...\times \Sigma_{d}. $$
 For $\Sigma_{i}$ we consider the cube $$ Q_{j}=\prod_{i=1}^{d}[y_{{i},j_{i}-1},y_{{i},{j_{i}}}),$$
 with volume
 $$ \Delta_{j}:= \prod_{i=1}^{d}(y_{j_{i}}- y_{{j_{i}-1}}).$$
 \end{definition}
 Now we introduce the notion of an admissible partition $\Sigma$ in $\mathbb R \times \mathbb R^{d}$ and the corresponding  $l^{p,q}(\Sigma)$-norm of a function $f:\mathbb R\times \mathbb R^{d} \to \mathbb C.$
 \begin{definition}
We say that a sequence $\Sigma=(x_{k},y_{j})_{{k\in \mathbb Z},{j\in \mathbb Z^{d}}}$ is an admissible sequence or partition if $\Sigma_{x}=(x_{k})_{k\in \mathbb Z}$ and $\Sigma_{y}=(y_{j})_{j\in \mathbb Z^d}$ are admissible partitions of $\mathbb R$ and $\mathbb R^{d}$ respectively.

    For the admissible partition  $\Sigma=(x_{k},y_{j})_{{k\in \mathbb Z},{j\in \mathbb Z^{d}}}$ we construct \\
    $$ Q_{jk}=[x_{k-1},x_{k})\times \prod_{i=1}^{d}[x_{{i},j_{i}-1},x_{{i},{j_{i}}})$$
    with volume $$ \Delta_{jk}:=(x_{k}-x_{k-1})\times \prod_{i=1}^{d}(x_{{i},j_{i}}- x_{{i},{j_{i}-1}}).$$
    The $l^{p,q}(\Sigma)$-norm of a function $f:\mathbb R\times \mathbb R^{d} \to \mathbb C$ is defined by
    $$\|f\|_{l^{p,q}(\Sigma)}:=\left\{ \sum_{k\in \mathbb Z}\left(
  \sum_{j\in \mathbb Z^{d}} \sup_{z\in Q_{jk}}|f(z)|^{q}\Delta_{jk} \right)^{\frac{p}{q}}\right\}^{\frac{1}{p}}$$
  where $z=(x,\bold y)\in \mathbb R \times \mathbb R^{d}$.


  \end{definition}
  
  \begin{definition}
      Let $1\leq p,q<\infty$. The $\Lambda^{p,q}$ space is defined by
      $$\Lambda^{p,q}=\left\{f\in M(\mathbb R\times \mathbb R^{d}): \|f\|_{l^{p,q}(\Sigma)}<\infty \right\}.$$
  \end{definition}

In this paper, we analyze the convergence behavior of linear operators in the settings of a mixed Lebesgue norm. In particular, we prove the convergence rate of the family of linear operators
$\{\mathscr{L}_{\eta}:\Lambda^{p,q}\longrightarrow L^{p,q}({\mathbb R \times \mathbb R})\}_{\eta \in \mathbb I},$ where $\mathbb I$ is an index set. Furthermore, as an application of our analysis of these linear operators, we obtain the rate of convergence of the following neural network sampling operators in mixed Lebesgue norm: Let $I=[-1,1]$. For $f:I\times I\to \mathbb{R},$ the neural network operators is given by 
\[  F_{n}f(x,y)=  \displaystyle\frac{ \displaystyle\sum_{k=-n}^{n}\sum_{j=-n}^{n}f{\left( \frac{k}{n}, \frac{j}{n}\right)}\psi_{\sigma}(nx-k, ny-j)}{ \displaystyle\sum_{k=-n}^{n}\sum_{j=-n}^{n}\psi_{\sigma}(nx-k, ny-j)},\]

where $\displaystyle \psi_{\sigma}(x,y):= \phi_{\sigma}(x)\times \phi_{\sigma}(y)$, $\phi_{\sigma}$ is defined in the following way:

\[ \phi_{\sigma}(v):= \frac{1}{2}[\sigma(v
+1)-\sigma(v-1)], \ v\in \mathbb{R}. \]

Here, the sigmoidal function $\sigma$ is defined in the following way: 
 \begin{definition}
      We say that a measurable function $\sigma:\mathbb{R}\to \mathbb{R}$ is a sigmoidal function if and only if $\lim_{x\to-\infty}\sigma(x)=0$ and $\lim_{x\to\infty}\sigma(x)=1.$
  \end{definition}

Further, $\sigma:\mathbb{R}\to \mathbb{R}$ is a non-decreasing sigmoidal function such that $\sigma(2)>\sigma(0)$ satisfies the following assumptions:
\begin{itemize}
    \item[($A_1$)] $\sigma(v)-\frac{1}{2}$ is an odd function;
    \item[($A_2$)] $\sigma\in C^{2}(\mathbb{R})$ is concave for $x\geq 0$;
    \item[($A_3$)] $ \sigma(v)=\mathcal{O}(|v|^{-1-\alpha})$ as $x\to -\infty$, for some $\alpha>0.$
\end{itemize}
The paper is organized as follows. Our main contribution in this paper is that we estimate the rate of approximation of certain linear operators in mixed Lebesgue norm. In order to do that in Sec \ref{section2}, we introduce the averaged modulus of smoothness for functions in $L^{p,q}(I),$ where $I\subset \mathbb{R}\times \mathbb{R}^{d}$ and analyze their properties in a detailed way. In Sec \ref{section3}, we consider the suitable subspace of $L^{p,q}(\mathbb{R}\times \mathbb{R}^{d})$ for which the averaged modulus of smoothness is finite. We prove several properties of these subspaces in this section. Using these results, we prove the rate of convergence of linear operators in-terms of averaged modulus of smoothness in Sec \ref{section4}. As an application of Sec \ref{section4}, we obtain the rate of approximation of neural network operators in term of averaged modulus of smoothness in Sec \ref{section5}. In last section, we discuss some examples of sigmoidal functions and using these sigmoidal functions, we show the implementation of continuous and discontinuous functions by neural network operators .

\section{The Averaged Modulus of Smoothness and its Properties}\label{section2}

In this section, we introduce and study the properties of averaged modulus of smoothness for functions in $L^{p,q}(I),$ where $I\subset \mathbb{R}\times \mathbb{R}^{d}.$



Further, let
\begin{eqnarray*}
    (\Delta_{h,\mathbf{h}}^{r}f)(x,\mathbf{y}):=\large \sum_{j=0}^{r}(-1)^{r-j}{\binom{r}{j}}f(x+jh,\mathbf{y}+j\mathbf{h}),
\end{eqnarray*}
where $x\in \mathbb{R},\mathbf{y}\in \mathbb{R}^{d},r\in \mathbb{N}$ and $\mathbf{h}=(h,.....,h)\in \mathbb{R}^{d}$.

Now we define the local modulus of smoothness for functions in $M\left([a,b]\displaystyle\times\prod_{i=1}^{d}[c_{i},d_{i}]\right)$ and $M(\mathbb{R}\times \mathbb{R}^{d}).$
Let $x\in [a,b], y_{i}\in [c_{i},d_{i}], 0\leq \delta^{*}\leq \left(\frac{b-a}{r}\right), 0\leq\delta_{i}\leq \left(\frac{d_{i}-c_{i}}{r}\right)$, for $i=1,2,...,d.$ Let us define $\delta:=\displaystyle \min_{i\in \{1,2,...,d\}}\{\delta^{*},\delta_{i}\}.$

\begin{definition}
    Let $\displaystyle f\in M\left([a,b]\times\prod_{i=1}^{d}[c_{i},d_{i}]\right).$ Then, the local modulus of smoothness is  defined by\\

    \noindent
    $\displaystyle\omega_{r}\left(f,(x,\mathbf{y});\delta;M\left([a,b]\times\prod_{i=1}^{d}[c_{i},d_{i}]\right)\right)$
    \begin{eqnarray*}
        &:=& \sup_{t, t+rh \in \left[x-\frac{r\delta}{2},x+\frac{r\delta}{2}\right]\bigcap[a,b]}\ \  \sup_{\bold{s},\bold{s}+r\bold{h}\in \left[\bold{y}-\frac{r\mathbf{\delta}}{2}, \bold{y}+\frac{r\mathbf{\delta}}{2}\right]\bigcap \prod_{i=1}^{d}[c_{i},d_{i}]} \left\{ |(\Delta_{h,\mathbf{h}}^{r}f)(t,\mathbf{s})| \right \}.
    \end{eqnarray*}

For $f\in M(\mathbb{R}\times \mathbb{R}^{d}),$ the local modulus of smoothness at a point $(x,\mathbf{y})\in \mathbb{R}\times \mathbb{R}^{d}$ for $\delta \ge 0$ is defined by\\

    \noindent
     $\omega_{r}(f,(x,\mathbf{y});\delta;M(\mathbb{R}\times \mathbb{R}^{d}))$
     \begin{eqnarray*}
         &:=&  \displaystyle \sup_{t, t+rh \in \left[x-\frac{r\delta}{2},x+\frac{r\delta}{2}\right]}\ \  \sup_{\bold{s},\bold{s}+r\bold{h}\in \left[\bold{y}-\frac{r\bold{\delta}}{2}, \bold{y}+\frac{r\bold{\delta}}{2}\right]} \left\{ |(\Delta_{h,\mathbf{h}}^{r}f)(t,\mathbf{s})|:\ \forall i\right \}.
     \end{eqnarray*}
\end{definition}
\begin{remark}
    For $f\in M(\mathbb{R}\times \mathbb{R}^{d})$, the local modulus of smoothness will coincide with the local modulus of smoothness of f defined on some finite $(k+1)$ cells with special boundary points. For all $\displaystyle(x,\mathbf{y})\in[a,b]\times\prod_{i=1}^{d}[c_{i},d_{i}]$\ and for all $c\ge \frac{r\delta}{2}$ with $\delta\ge 0$ we have,\\
    $\displaystyle \omega_{r}\left(f,(x,\mathbf{y});\delta;M\left([a-c,b-c]\times\prod_{i=1}^{d}[c_{i}-c,d_{i}-c]\right)\right)=\omega_{r}\left(f;(x,\mathbf{y});\delta;M(\mathbb{R}\times \mathbb{R}^{d})\right).$
\end{remark}
The following observations can be made from the definition of local modulus of smoothness.
\begin{itemize}
    \item [(i)] $\omega_{r}(f,(x,\mathbf{y});\delta;M(\mathbb{R}\times \mathbb{R}^{d}))\leq 2^{r} \displaystyle \sup_{(x,\mathbf{y})\in \mathbb{R}\times \mathbb{R}^{d}}|f(x,\mathbf{y})|$.\\

    \item [(ii)] For $\delta \leq \delta^{\prime}$ with $\delta$ as given previously and $\delta^{\prime}= \displaystyle \min_{i\in \{1,2,...,d\}}\{{\delta^{\prime}}^{*},\delta^{\prime}_{i}\}.$\\
    $\omega_{r}(f,(x,\bold y);\delta;M(\mathbb{R}\times \mathbb{R}^{d}))\leq \omega_{r}(f,(x,\bold y);\delta^{\prime};M(\mathbb{R}\times \mathbb{R}^{d})).$
\end{itemize}
Now we define the $L^{p,q}$-averaged modulus of smoothness for functions in\\
$M\left([a,b]\displaystyle\times\prod_{i=1}^{d}[c_{i},d_{i}]\right)$ and $M(\mathbb{R}\times \mathbb{R}^{d}).$
\begin{definition}
  Let $1\leq p,q <\infty$ and $f\in M\left([a,b]\times \displaystyle \prod_{i=1}^{d}[c_{i},d_{i}]\right).$ Then, for $\  0\leq \delta\leq \displaystyle\inf_{i=1,2,...,d}\left\{\frac{(b-a)}{r},\frac{(d_{i}-c_{i})}{r}\right\}$ the $L^{p,q}$-averaged modulus of smoothness of order $r\in \mathbb N$ is defined by
  \\
  \noindent
  $\displaystyle \tau_{r}\left(f;\delta;M\left([a,b]\times\prod_{i=1}^{d}[c_{i},d_{i}]\right)\right)_{p,q}$
   \begin{eqnarray*}
    &:=& \left\|\displaystyle \omega_{r}\left(f;.;\delta;M([a,b]\times \prod_{i=1}^{d}[c_{i},d_{i}]\right)\right\|_{L^{p,q}\left([a,b]\times\prod_{i=1}^{d}[c_{i},d_{i}]\right)}\\ &=&\displaystyle \left(\int_{[a,b]}\left(\int_{\prod_{i=1}^{d}[c_{i},d_{i}]}\left|\omega_{r}(f;.;\delta; M([a,b]\times\prod_{i=1}^{d}[c_{i},d_{i}])\right|^{q}d\bold y\right )^{\frac{p}{q}}dx\right)^{\frac{1}{p}}.
   \end{eqnarray*}

    Let $f\in M(\mathbb{R}\times \mathbb{R}^{d})$ and $\delta\ge 0$. Then, the averaged modulus of smoothness is defined by
    \begin{eqnarray*}
        \tau_{r}(f;\delta;M(\mathbb{R}\times \mathbb{R}^{d}))_{p,q}
   &:= & \| \omega_{r}(f;.;\delta; M(\mathbb{R}\times \mathbb{R}^{d})\|_{L^{p,q}(\mathbb{R}\times \mathbb{R}^{d})}\\
   &=&\displaystyle \left(\int_{\mathbb{R}}\left(\int_{\mathbb{R}^{d}}\left|\omega_{r}(f;.;\delta; M(\mathbb{R}\times \mathbb{R}^{d})\right|^{q}d\bold y\right )^{\frac{p}{q}}dx\right)^{\frac{1}{p}}.
    \end{eqnarray*}
\end{definition}

\begin{prop}\label{P1}
    Let $f\in M(\mathbb{R}\times \mathbb{R}^{d}),\ r\in \mathbb N,  1\leq p, q< \infty.$ Then, we have
$$\displaystyle \lim_{n\to\infty}\tau_{r}\left(f;\delta;M\left([-n,n]\times[-n,n]^{d}\right)\right)_{p,q}= \tau_{r}(f;\delta;M(\mathbb{R}\times \mathbb{R}^{d}))_{p,q}.$$
\end{prop}

\begin{proof}
 Let $f\in M(\mathbb{R}\times \mathbb{R}^{d})$ and $\delta\ge0$. For $ n\ge r\delta$ and $j\in \mathbb N$ with $j\ge \frac{r\delta}{2}$ we have\\

 \noindent
 $\displaystyle\omega_{r}\left(f,(x,\mathbf{y});\delta;M\left([-n,n]\times[-n,n]^{d}\right)\right)$
 \begin{eqnarray*}
         &\leq&  \omega_{r}(f,(x,\mathbf{y});\delta;M(\mathbb{R}\times \mathbb{R}^{d}))\\
        &=&\displaystyle\omega_{r}\left(f,(x,\mathbf{y});\delta;M\left([-n-j,n+j]\times[-n-j,n+j]^{d}\right)\right).
 \end{eqnarray*}
 Now taking the $L^{p,q}$ norm on both side, we obtain\\

 \noindent
$ \tau_{r}\left(f;\delta;M\left([-n,n]\times[-n,n]^{d}\right)\right)_{p,q}$
   \begin{eqnarray*}
       &=& \left|\left| \displaystyle\omega_{r}\left(f,(x,\mathbf{y});\delta;M\left([-n,n]\times[-n,n]^{d}\right)\right)\right|\right|_{L^{p,q}}\\
  &\leq& \left|\left| \displaystyle\omega_{r}\left(f,(x,\mathbf{y});\delta;M\left(\mathbb R\times \mathbb R^{d}\right)\right)\right|\right|_{L^{p,q}}\\
  &\leq& \left|\left| \displaystyle\omega_{r}\left(f,(x,\mathbf{y});\delta;M\left([-n-j,n+j]\times[-n-j,n+j]^{d}\right)\right)\right|\right|_{L^{p,q}}\\
   &\leq&  \displaystyle\tau_{r}\left(f;\delta;M\left([-n-j,n+j]\times[-n-j,n+j]^{d}\right)\right)_{{p,q}}.
      \end{eqnarray*}
We obtain the desired estimate by taking the limit as $n\to \infty$. Hence, the proof is completed.
   \end{proof}

  \begin{lemma}
       Let $\displaystyle f\in M\left([a,b]\times\prod_{i=1}^{d}[c_{i},d_{i}]\right),\ 1\leq q<p< \infty,\ \frac{1}{p}+ \frac{1}{p^{\prime}}=1,\ \frac{1}{q}+ \frac{1}{q^{\prime}}=1.$ Then, we have
      \begin{itemize}
          \item [(i)] $\displaystyle \tau_{r}\left(f;\delta;M\left([a,b]\times\prod_{i=1}^{d}[c_{i},d_{i}]\right)\right)_{1,1}\leq K_{1} \displaystyle \tau_{r}\left(f;\delta;M\left([a,b]\times\prod_{i=1}^{d}[c_{i},d_{i}]\right)\right)_{p,q}$
          \item [(ii)] $\displaystyle \tau_{r}\left(f;\delta;M\left([a,b]\times\prod_{i=1}^{d}[c_{i},d_{i}]\right)\right)_{p,q} \leq K_{2}\displaystyle \left\{\tau_{r}\left(f;\delta;M\left([a,b]\times\prod_{i=1}^{d}[c_{i},d_{i}]\right)\right)_{1,1}\right\}^{\frac{1}{p}},$
      \end{itemize}
     where
     \begin{eqnarray*}
        \displaystyle K_{1}&:=& (b-a)^{\frac{1}{p^{\prime}}} \left(\prod_{i=1}^{d}(d_{i}-c_{i})\right)^{\frac{1}{q^{\prime}}}\\
       K_{2}&:=&\left \{\displaystyle 2^{r}\sup_{(x,\bold y)\in [a,b]\times\prod_{i=1}^{d}(d_{i}-c_{i})}|f(x,\bold{y})|\right\}^{\frac{1}{p^{\prime}}}.
     \end{eqnarray*}
      Furthermore, from (i) and (ii), we have
      \begin{gather*}
       \displaystyle \lim_{\delta\to 0+} \tau_{r}\left(f;\delta;M\left([a,b]\times\prod_{i=1}^{d}[c_{i},d_{i}]\right)\right)_{p,q}=0 \\
          \Updownarrow \\
        \displaystyle \lim_{\delta\to 0+} \tau_{r}\left(f;\delta;M\left([a,b]\times\prod_{i=1}^{d}[c_{i},d_{i}]\right)\right)_{1,1}=0.
      \end{gather*}
   \end{lemma}
\begin{proof}
   (a) The proof of the proposition follows from the following inequalities
    \begin{itemize}
        \item[(a)]  $\|g\|_{1,1}\leq (b-a)^{\frac{1}{p^{\prime}}} \left(\prod_{i=1}^{d}(d_{i}-c_{i})\right)^{\frac{1}{q^{\prime}}}\|g\|_{p,q} $\\

        \item[(b)] $\|g\|_{p,q}\leq \left \{\displaystyle\sup_{(x,\bold y)\in [a,b]\times\prod_{i=1}^{d}(d_{i}-c_{i})}|g(x,\bold{y})|\right\}^{\frac{1}{p^{\prime}}}\|g\|_{1,1}^{\frac{1}{p}}.$
    \end{itemize}
    By applying Holder's inequality twice, we obtain the inequality $(a)$. Now, we prove the second inequality. By using the definition of mixed Lebesgue norm, we have
    $$ \|g\|_{p,q}= \displaystyle \left(\int_{[a,b]}\left(\int_{\prod_{i=1}^{d}[c_{i},d_{i}]}\left|g(x,\bold y)\right|^{q}d\bold y\right )^{\frac{p}{q}}dx\right)^{\frac{1}{p}}.$$
    Now, using Jensen's inequality, we obtain
     \begin{eqnarray*}
          \|g\|_{p,q} &\leq& \displaystyle \left(\int_{[a,b]}\left(\int_{\prod_{i=1}^{d}[c_{i},d_{i}]}\left|g(x,\bold y)\right|^{q.{\frac{p}{q}}}d\bold y\right )dx\right)^{\frac{1}{p}}\\
          &\leq& \left \{\displaystyle\sup_{(x,\bold y)\in [a,b]\times\prod_{i=1}^{d}(d_{i}-c_{i})}|g(x,\bold{y})|\right\}^{\frac{p-1}{p}}\displaystyle \left(\int_{[a,b]}\left(\int_{\prod_{i=1}^{d}[c_{i},d_{i}]}\left|g(x,\bold y)\right|d\bold y \right)dx\right )^{\frac{1}{p}}\\
           &\leq& \left \{\displaystyle\sup_{(x,\bold y)\in [a,b]\times\prod_{i=1}^{d}(d_{i}-c_{i})}|g(x,\bold{y})|\right\}^{\frac{1}{p^{\prime}}}\|g\|_{1,1}^{\frac{1}{p}}.
     \end{eqnarray*}
     Hence, we get the required estimate.
\end{proof}
In a similar way, we have the following lemma for $f\in M(\mathbb R \times \mathbb R^{d})$.
\begin{lemma}
  Let $f\in M(\mathbb R \times \mathbb R^{d}),\ 1\leq q<p< \infty,\ \frac{1}{p}+ \frac{1}{p^{\prime}}=1,\ \frac{1}{q}+ \frac{1}{q^{\prime}}=1.$ Then, we have
  $$ \tau_{r}(f;\delta;M(\mathbb R \times \mathbb R^{d}))_{p,q}\leq K \tau_{r}(f;\delta;M(\mathbb R \times \mathbb R^{d}))_{1,1},$$
  where $\displaystyle K=\left( 2^{r}\sup_{(x,\bold y)\in \mathbb R \times \mathbb R^{d} }|f(x,\bold y)| \right)^{\frac{1}{p^{\prime}}}$.
  Furthermore, we have
  $$ \lim_{\delta \to 0+}\tau_{r}(f;\delta;M(\mathbb R \times \mathbb R^{d}))_{1,1}=0 \implies  \lim_{\delta \to 0+} \tau_{r}(f;\delta;M(\mathbb R \times \mathbb R^{d}))_{p,q}=0$$
\end{lemma}
\begin{proof}
    Following the lines of the above lemma and using $Proposition 1 $ we can easily establish the lemma.
\end{proof}
  \begin{prop}
      Let $1\leq p,q< \infty$ and $f\in M\left([a,b]\times\prod_{i=1}^{d}[c_{i},d_{i}]\right) $. Then f is Riemann integrable if and only if
      $$\displaystyle \lim_{\delta\to 0+} \tau_{r}\left(f;\delta;M\left([a,b]\times\prod_{i=1}^{d}[c_{i},d_{i}]\right)\right)_{p,q}=0.$$
  \end{prop}
  \begin{proof}
      By lemma 1, it is sufficient to show that
      $f$ is Riemann integrable if and only if \\
      $\displaystyle \lim_{\delta\to 0+} \tau_{r}\left(f;\delta;M\left([a,b]\times\prod_{i=1}^{d}[c_{i},d_{i}]\right)\right)_{1,1}=0.$ \\
      As $L^{1,1}$ norm on $\mathbb{R} \times \mathbb{R}^{d} $ acts like an $L^{1}$ norm on $\mathbb{R}^{d+1}$ and the rest of the proof follows in the same line as in \cite{mulima}.
  \end{proof}
  Now, we define the integral moduli of smoothness in the mixed Lebesgue spaces.
  \begin{definition}
  \begin{itemize}
     \item [(i)] Let $1\leq p,q< \infty$ and $f\in;L^{p,q}\left((a,b)\times\prod_{i=1}^{d}(c_{i},d_{i})\right).$Then the integral moduli of smoothness is defined by\\
     \noindent
     $\displaystyle \omega_{r}(f;\delta;L^{p,q}\left((a,b)\times\prod_{i=1}^{d}(c_{i},d_{i})\right)$
    \begin{eqnarray*}
       :=\displaystyle \sup_{0\leq h\leq \delta}\left\{ \int_{(a,b-rh)}\left(\int_{\prod_{i=1}^{d}(c_{i},d_{i}-rh)} |(\Delta_{h,\bold{h}}^{r}f)(u,\bold v)|^{q} d\bold v \right)^{\frac{p}{q}}du \right\}^{\frac{1}{p}}.
    \end{eqnarray*}

     \item [(ii)] Similarly for $f\in L^{p,q}(\mathbb{R} \times \mathbb{R}^{d})$ the integral moduli of smoothness is defined by
         $$\omega_{r}(f;\delta;L^{p,q}(\mathbb{R} \times \mathbb{R}^{d})= \displaystyle \sup_{0\leq h\leq \delta}\left\{ \int_{\mathbb{R}}\left(\int_{\mathbb{R}^{d}} |(\Delta_{h,\bold{h}}^{r}f)(u,\bold v)|^{q} d\bold v \right)^{\frac{p}{q}}du \right\}^{\frac{1}{p}}.$$

  \end{itemize}
     \end{definition}

\begin{prop}
   If $f \in L^{p,q}(\mathbb{R} \times \mathbb{R}^{d}), r\in \mathbb N $ and  $1\leq p,q \leq \infty,$ then we have
    $$ \lim_{n\to \infty}\omega_{r}(f;\delta;L^{p,q}\left((-n,n)\times(-n,n)^{d}\right)= \omega_{r}(f;\delta;L^{p,q}(\mathbb{R}\times \mathbb{R}^{d})).$$
    \end{prop}

\begin{proof}
    From the definition of integral moduli of smoothness, we have
    $$\omega_{r}(f;\delta;L^{p,q}\left((-n,n)\times(-n,n)^{d}\right)\leq \omega_{r}(f;\delta;L^{p,q}(\mathbb{R}\times \mathbb{R}^{d})), \forall n\in \mathbb N,$$
   which implies that
   \begin{equation}
     \lim_{n\to \infty}\omega_{r}(f;\delta;L^{p,q}\left((-n,n)\times(-n,n)^{d}\right)\leq \omega_{r}(f;\delta;L^{p,q}(\mathbb{R}\times \mathbb{R}^{d})).
   \end{equation}
    Again, from the definition of integral moduli of smoothness for $0\leq h\leq \delta$, we obtain \\
    \noindent
    $\displaystyle \lim_{n\to \infty}\omega_{r}(f;\delta;L^{p,q}\left((-n,n)\times(-n,n)^{d}\right)$
    \begin{eqnarray*}
      &=&\lim_{n\to \infty}\displaystyle \sup_{0\leq h\leq \delta}\left\{ \int_{(-n,n-rh)}\left(\int_{(-n,n-rh)^{d}} |(\Delta_{h,\bold{h}}^{r}f)(u,\bold v)|^{q} d\bold v \right)^{\frac{p}{q}}du \right\}^{\frac{1}{p}}\\
      &\geq& \lim_{n\to \infty} \left\{ \int_{(-n,n-rh)}\left(\int_{(-n,n-rh)^{d}} |(\Delta_{h,\bold{h}}^{r}f)(u,\bold v)|^{q} d\bold v \right)^{\frac{p}{q}}du \right\}^{\frac{1}{p}}\\
      &=& \left\{ \int_{\mathbb{R}}\left(\int_{\mathbb{R}^{d}} |(\Delta_{h,\bold{h}}^{r}f)(u,\bold v)|^{q} d\bold v \right)^{\frac{p}{q}}du \right\}^{\frac{1}{p}}= \omega_{r}(f;\delta;L^{p,q}(\mathbb{R}\times \mathbb{R}^{d})),
    \end{eqnarray*}
    which implies
    \begin{equation}
       \omega_{r}(f;\delta;L^{p,q}(\mathbb{R}\times \mathbb{R}^{d})) \leq \omega_{r}(f;\delta;L^{p,q}\left((-n,n)\times(-n,n)^{d}\right)
    \end{equation}
    We obtain the desired result by combining the estimates (2.1) and (2.2).
\end{proof}
\begin{prop}
    Let $f\in M\left((a,b)\times\prod_{i=1}^{d}(c_{i},d_{i})\right)$ , $r\in \mathbb N $ and  $1\leq p,q \leq \infty.$ Then we have
    $$\displaystyle \omega_{r}\left(f;\delta;L^{p,q}\left((a,b)\times\prod_{i=1}^{d}(c_{i},d_{i})\right)\right)\leq \tau_{r}\left(f;\delta;M\left((a,b)\times\prod_{i=1}^{d}(c_{i},d_{i})\right)\right)_{p,q}. $$

    Similarly, for $f\in M(\mathbb{R}\times \mathbb{R}^{d})\cap L^{p,q}(\mathbb{R}\times \mathbb{R}^{d})$ we have
    $$\displaystyle \omega_{r}(f;\delta;L^{p,q}(\mathbb{R}\times \mathbb{R}^{d})\leq \tau_{r}(f;\delta;M(\mathbb{R}\times \mathbb{R}^{d})_{p,q}. $$
\end{prop}
\section{Subspaces of $L^{p,q}(\mathbb{R}\times \mathbb{R}^{d})$}\label{section3}

In this section, we prove various properties of the subspaces of $L^{p,q}(\mathbb{R}\times \mathbb{R}^{d}).$ 
  
  \begin{lemma}
      A function $f\in M(\mathbb R\times \mathbb R^{d})\in \Lambda^{p,q}$ if and only if there exist a $\Delta_{*}>0$ such that $\|f\|_{l^{p,q}(\Sigma)}<\infty$ for each admissible sequence $\Sigma$ with lower mesh size $\underline{\Delta}\geq \Delta_{*}$.
  \end{lemma}

\begin{proof}
    By the definition of $\Lambda^{p,q}$, we see that if $f\in \Lambda^{p,q}$ then  $\|f\|_{l^{p,q}(\Sigma)}<\infty$ for any admissible partition $\Sigma$ so we are done. Now, We show the converse part.\\
     Let $\Sigma=(x_{k},y_{j})_{{k\in \mathbb Z},{j\in \mathbb Z^{d}}}$ be an arbitrary admissible sequence with lower mesh size $\underline{\Delta}>\Delta_{*}$. Now we can choose $m\in \mathbb N$ such that $m\underline{\Delta}\geq \Delta_{*}$. So we have
      $$\Delta_{*}\leq m\underline{\Delta}\leq (x_{mk}-x_{m(k-1)}) \leq m \bar{\Delta} $$
      $$\Delta_{*}\leq m\underline{\Delta}\leq (y_{i,mj_{i}}-y_{i,m(j_{i}-1)}) \leq m \bar{\Delta}, \ \forall i=1,2,...,d.$$
      From the above estimates, we see that the subsequence $(x_{mk},y_{mj})_{{k\in \mathbb Z},{j\in \mathbb Z^{d}}}$ has lower mesh size greater than equal to $\Delta_{*}$. Furthermore, for $N=(N_{0},N_{y})$ where $N_{y}=(N_{1},...,N_{d})$, we get
      \begin{eqnarray}
          \sum_{\bold j\in \mathbb Z^{d}}\sup_{\bold y\in Q_{j}}|f(z)|^{q}=\sum_{N_{1}=0}^{m-1}...\sum_{N_{d}=0}^{m-1}\sum_{j\in \mathbb Z^{d}}\sup_{\bold y\in Q_{mj-N_{y}}}|f(z)|^{q}.
      \end{eqnarray}
      For all such $N_{y}$'s we have $Q_{mj-N_{y}}\subseteq \widetilde{Q_{mj}}:=\prod_{i=1}^{d}[x_{{i},m(j_{i}-1)+N_{i}},x_{{i},{mj_{i}}}). $
      Now, from the above relation$(3.1)$, we get
      $$ \sum_{\bold j\in \mathbb Z^{d}}\sup_{\bold y\in Q_{j}}|f(z)|^{q}\leq\sum_{N_{1}=0}^{m-1}...\sum_{N_{d}=0}^{m-1}\sum_{j\in \mathbb Z^{d}}\sup_{\bold y\in\widetilde{Q_{mj}}}|f(z)|^{q}=m^{d}\sum_{\bold j\in \mathbb Z^{d}}\sup_{\bold y\in\widetilde{Q_{mj}}}|f(z)|^{q}.$$
      For $\infty>p\geq q\geq 1$, we obtain
      $$  \left(\sum_{\bold j\in \mathbb Z^{d}}\sup_{\bold y\in Q_{j}}|f(z)|^{q}\right)^{\frac{p}{q}}\leq m^{\frac{dp}{q}}\left(\sum_{\bold j\in \mathbb Z^{d}}\sup_{\bold y\in\widetilde{Q_{mj}}}|f(z)|^{q}\right)^{\frac{p}{q}}.$$
      Hence, we have
      $$ \sum_{k\in \mathbb Z}\left(
  \sum_{j\in \mathbb Z^{d}} \sup_{z\in Q_{jk}}|f(z)|^{q}\Delta_{jk} \right)^{\frac{p}{q}}\leq m^{\frac{dp}{q}+1}\sum_{k\in \mathbb Z}\left(\sum_{\bold j\in \mathbb Z^{d}}\sup_{\bold y\in\widetilde{Q_{mjk}}}|f(z)|^{q}\right)^{\frac{p}{q}}< \infty ,$$
  where $\widetilde{Q_{mjk}}=\widetilde{Q_{mj}}\times (x_{(mk-1)+N_{0}},x_{mk})$ and these cubes corresponds to the admissible sequence $(x_{mk},y_{mj})_{{k\in \mathbb Z},{j\in \mathbb Z^{d}}}$. By the definition of $\Lambda^{p,q}$, we obtain the converse part for $p\geq q$. Now, from \cite{gord}, we have if $p_{1}>p_{2}>0$ then $ \|f\|_{L^{p_{2},p_{1}}}\leq \|f\|_{L^{p_{1},p_{2}}}$
  so we get the result when $p<q$. Thus, the proof is completed.
\end{proof}
\begin{prop}
  If $1\leq p, q< \infty$, then  $\Lambda^{p,q}$ is a proper subspace of $L^{p,q}(\mathbb R\times \mathbb R^{d})$.
\end{prop}
\begin{proof}
For $f(x,\bold y)\in \Lambda^{p,q}$, we have to show that
  $$\left(\int_{\mathbb{R}}\left(\int_{\mathbb{R}^{d}}|f(x,\mathbf{y})|^{q}d\bold y\right )^{\frac{p}{q}}dx\right)^{\frac{1}{p}}<+\infty.$$
  Let us take the admissible sequence as $\Sigma=(x_{k}, \bold y_{\bold j})_{k\in \mathbb Z, \bold j \in \mathbb Z^{d}}$, where $x_{k}=k$ and $ y_{\bold j}=\bold j$, $\bold j=(j_{1},j_{2},...,j_{d})$. Then the corresponding cube is given by  $$Q_{\bold {j}k}=[k-1,k]\prod_{i=1}^{d}[j_{i}-1,j_{i}]$$
  with volume $\Delta_{\bold {j}k}=1$ for all $k\in \mathbb Z, \ \bold j \in \mathbb Z^{d}$.\\
 For $p>q\geq 1$, we have
 $$\left(\int_{\mathbb{R}^{d}}|f(x,\mathbf{y})|^{q}d\bold y\right )^{\frac{p}{q}}\leq \left(\sum_{\bold j\in \mathbb Z^{d}}\sup_{\bold y \in Q_{\bold j}}|f(x,\mathbf{y})|^{q}d\bold y\right )^{\frac{p}{q}}.$$
 In view of $\Lambda^{p,q}$ and the above inequality, we obtain
  $$\int_{\mathbb{R}}\left(\int_{\mathbb{R}^{d}}|f(x,\mathbf{y})|^{q}d\bold y\right )^{\frac{p}{q}}dx\leq \sum_{k\in Z}\left(\sum_{\bold j\in \mathbb Z^{d}}\sup_{\bold y \in Q_{\bold j}k}|f(x,\mathbf{y})|^{q}d\bold y\right )^{\frac{p}{q}}< \infty. $$
Following (corollary 3.5.2) in \cite{gord}, we obtain the result for $p<q$. This completes the proof.
\end{proof}
Now we want to find some suitable subspaces of $L^{p,q}(\mathbb R \times \mathbb R^{d})$ for which the $\tau-$modulus is finite. We consider the following subspaces.

\begin{definition}
Let $1\leq p, q< \infty$ and $ r\in \mathbb N$. Then we define the following spaces:\\
\begin{itemize}
    \item [(i)] The space $\mathcal{F}^{p,q}$  by

\begin{equation*}
\begin{split}
\mathcal{F}^{p,q}:=\Bigg\{& f\in M(\mathbb R \times \mathbb R^{d}): \mbox{there exists} \ \eta, \zeta \ \mbox{such that} \\ & f(x,\bold{y})=\mathcal{O}\left({(1+|x|)^{(\frac{-1}{p}-\eta)}} {{(1+\|\bold y\|)}^{(\frac{-1}{q}-\zeta)}}\right) \Bigg\}.
\end{split}
\end{equation*}

  \item [(ii)] The space $\Omega^{p,q} $ by
\begin{equation*}
\begin{split}
\Omega^{p,q}:=\Bigg \{&f\in M(\mathbb R \times \mathbb R^{d}):|f(x,\bold{y})| \leq g(x,\bold{y}),\ g\in L^{p,q}(\mathbb R \times \mathbb R^{d}), \\ &  g\geq 0,\  \mbox{even and non-increasing}  \Bigg\}.
\end{split}
\end{equation*}

   \item [(iii)] The Sobolev spaces $W^{r} (L^{p,q}(\mathbb R \times \mathbb R^{d}))$ by
\begin{equation*}
\begin{split}
W^{r} (L^{p,q}(\mathbb R \times \mathbb R^{d})):=\Bigg\{& f\in L^{p,q}(\mathbb R \times \mathbb R^{d}): f(x,\bold{y})=\phi(x,\bold{y})\ a.e., \phi\in AC_{loc}^{r}(\mathbb R \times \mathbb R^{d})\  \\ & \mbox{and} \ D^{\beta}\phi\in L^{p,q}(\mathbb R \times \mathbb R^{d}) \forall |\beta|\leq r   \Bigg\},
\end{split}
\end{equation*}
   where $\ D^{\beta}=\frac{\partial^{n_{0}+|N|}}{\partial x^{n_{0}}\partial y_{1}^{n_{1}}\partial y_{2}^{n_{2}}...\partial y_{d}^{n_{d}}}$, $\beta=(n_{0},N),N=(n_{1},n_{2}...,n_{d})$ and $|\beta|=\sum_{i=1}^{d}n_{i}.$
\end{itemize}
\end{definition}

\begin{prop}
The following statement holds
    \begin{itemize}

                \item[(i)]  $\Omega^{p,q}$ and $W^{r} (L^{p,q}(\mathbb R \times \mathbb R^{d}))$ are linear spaces.

        \item[(ii)] $\mathcal{F}^{p,q}\subsetneq \Omega^{p,q}\subsetneq \Lambda^{p,q} $

    \end{itemize}
\end{prop}
\begin{proof}
    By the definition of $\Omega^{p,q}$ and $W^{r} (L^{p,q}(\mathbb R \times \mathbb R^{d}))$ it is clear that these two spaces are linear.
   Let $ f(x,\bold y)\in \mathcal{F}^{p,q}$. Then one can see that $|f(x,\bold y)|\in L^{p,q}(\mathbb R \times \mathbb R^{d})$ and hence $ f\in \Omega^{p,q}.$ Let $f \in \Omega^{p,q}$. Then by definition of $\Omega^{p,q}$, $f \in  M(\mathbb R \times \mathbb R^{d})$ and  there exist a function $ g\in L^{p,q}(\mathbb R \times \mathbb R^{d})$ such that $|f(x,\bold y)|\leq g(x,\bold y)$. Now let $\Sigma=(x_{k},y_{j})_{{k\in \mathbb Z},{j\in \mathbb Z^{d}}}$ be an arbitrary admissible partition of $\mathbb R \times \mathbb R^{d}$. We have to show that $\|f\|_{l^{p,q}(\Sigma)}<\infty$. For $p\geq q$, we have

    \begin{eqnarray*}
        \sum_{k\in \mathbb Z}\left(
  \sum_{j\in \mathbb Z^{d}} \sup_{z\in Q_{jk}}|f(z)|^{q}\Delta_{jk} \right)^{\frac{p}{q}}
  &\leq& \sum_{k\in \mathbb Z}\left(
  \sum_{j\in \mathbb Z^{d}} \sup_{z\in Q_{jk}}|g(z)|^{q}\Delta_{jk} \right)^{\frac{p}{q}}\\
   &\leq& (\bar{\Delta})^{\frac{p}{q}}\sum_{k\in \mathbb Z}\left(
  \sum_{j\in \mathbb Z^{d}} \sup_{z\in Q_{jk}}|g(z)|^{q} \right)^{\frac{p}{q}}\\
   &\leq& (\bar{\Delta})^{\frac{p}{q}} \int_{\mathbb R}\left(\int_{\mathbb R^{d}}|g(z)|^{q}d\bold{y} \right)^{\frac{p}{q}}dx< \infty.
    \end{eqnarray*}
Similar result is true for $p\leq q$  by Corollary 3.5.2 in \cite{gord}. This implies that $\Omega^{p,q}\subset \Lambda^{p,q}$.

\end{proof}

\begin{prop}
    Let $f\in \Lambda^{p,q}$, $1\leq p, q< \infty$ and $r\in \mathbb N$. Then, we have $$\omega_{r}(f,(x,.;\delta;M(\mathbb{R}\times \mathbb{R}^{d}))\in \Lambda^{p,q},\  \mbox{for every}\ \delta \ \geq 0.$$
\end{prop}

\begin{proof}
    Let $\Sigma=(x_{k}, \bold y_{\bold j})_{k\in \mathbb Z, \bold j \in \mathbb Z^{d}}$ be an admissible partition with  $\underline{\Delta}\geq \frac{r\delta}{2}$. From Lemma $2$, it is sufficient to show that
    $$ \|\omega_{r}(f,.;\delta;M(\mathbb{R}\times \mathbb{R}^{d}))\|_{l^{p,q}(\Sigma)}< \infty.$$
    Now for $z=(x,\bold y)$, $x\in \mathbb R$ and $y=(y_{1},...,y_{d})\in \mathbb R^{d}$, we consider
    $$ C_{\delta}(z)= \left[x-\frac{r\delta}{2}, x+\frac{r\delta}{2}\right]\times \prod_{i=1}^{d}\left[y_{i}-\frac{r\delta}{2}, y_{i}+\frac{r\delta}{2}\right].$$
 Then, by property (1) of the local modulus of smoothness, we have
 $$\omega_{r}(f,z;\delta;M(\mathbb{R}\times \mathbb{R}^{d}))\leq 2^{r}\sup_{\eta\in C_{\delta}(z)}|f(\eta)|.  $$
There exists $P=(p_{0},\bold p) $, $\bold p = (p_{1}, ..., p_{d})$ such that
 $$ C_{\delta}(z)\subset [x_{k-2},x_{k+1})\times \prod_{i=1}^{d}[y_{i,j_{i}-2},y_{i,j_{i}+1})= \bigcup_{p_{0}=-1}^{1}\bigcup_{p_{1}=-1}^{1}...\bigcup_{p_{d}=-1}^{1}Q_{jk+P}. $$
 Let us consider the case $p\geq q$. We have
 \begin{eqnarray*}
     &&\sum_{k\in \mathbb Z}\left(
  \sum_{j\in \mathbb Z^{d}} \sup_{z\in Q_{jk}}|\omega_{r}(f,z;\delta;M(\mathbb{R}\times \mathbb{R}^{d}))|^{q} \right)^{\frac{p}{q}}\\
  &\leq& 2^{rp}\sum_{k\in \mathbb Z}\left(
  \sum_{j\in \mathbb Z^{d}}\sup\left\{ |f(\eta)|^{q}: \eta\in \bigcup_{p_{0}=-1}^{1}\bigcup_{p_{1}=-1}^{1}...\bigcup_{p_{d}=-1}^{1}Q_{jk+P}   \right\}  \right)^{\frac{p}{q}}\\
   &\leq& 2^{rp}\sum_{p_{0}=-1}^{1}\sum_{p_{1}=-1}^{1}...\sum_{p_{d}=-1}^{1} \sum_{k\in \mathbb Z}\left(
  \sum_{j\in \mathbb Z^{d}} \sup_{\eta\in Q_{jk+P}}|f(\eta)|^{q} \right)^{\frac{p}{q}}\\
  &\leq& 2^{rp} \ 3^{d+1}\sum_{k\in \mathbb Z}\left(
  \sum_{j\in \mathbb Z^{d}} \sup_{\eta\in Q_{jk}}|f(\eta)|^{q} \right)^{\frac{p}{q}}< \infty.
 \end{eqnarray*}
Similarly, the result is true for $p\leq q$  by Corollary 3.5.2 in \cite{gord}. This completes the proof of the assertion.
\end{proof}

\begin{corollary}
 If $f\in \Lambda^{p,q}$, $1\leq p, q< \infty$, $\delta\geq 0$ and $r\in \mathbb N$, then we have
 $$\tau_{r}(f;\delta;M(\mathbb{R}\times \mathbb{R}^{d}))_{p,q} < \infty.$$
\end{corollary}

\begin{proof}
    The proof straightly follows from the above proposition.
\end{proof}
\begin{definition} The space $R_{loc}(\mathbb R\times \mathbb R^{d})$ is defined by
   $$ R_{loc}(\mathbb R\times \mathbb R^{d}):=\{ f:\mathbb R\times \mathbb R^{d}\to \mathbb C; \mbox{f is locally Riemann integrable on}\ \mathbb R\times \mathbb R^{d} \}.$$
\end{definition}

\begin{prop}
    For $f\in \Lambda^{p,q}\cap R_{loc}(\mathbb R\times \mathbb R^{d})$, $1\leq p, q< \infty$, $r\in \mathbb N$, we have
    $$ \lim_{\delta\to 0}\tau_{r}(f;\delta;M(\mathbb{R}\times \mathbb{R}^{d}))_{p,q} =0.$$
\end{prop}

\begin{proof}
   In view of Lemma 2, it is enough to show the result holds for $L^{1,1}$-norm using $Lemma \ 2$. Let $\epsilon>0$ be given. For large $N>0$, let $A=(\mathbb N, \infty)\cup (-\infty, -\mathbb N)$, we have
    \begin{eqnarray}
       \int_{A}\int_{A^{d}} \omega_{r}(f,(x,\mathbf{y});\delta;M(\mathbb{R}\times \mathbb{R}^{d})) d\bold y dx < \epsilon.
    \end{eqnarray}
    Then for all $0\leq \delta\leq \frac{2}{r}$, we get
    \begin{eqnarray*}
        \tau_{r}(f;\delta;M(\mathbb{R}\times \mathbb{R}^{d}))_{1,1}&=& \int_{[-\mathbb N, \mathbb N]}\int_{[-\mathbb N, \mathbb N]^{d}} \omega_{r}(f,(x,\mathbf{y});\delta;M(\mathbb{R}\times \mathbb{R}^{d})) d\bold y dx\\
        &+&\ \int_{A}\int_{A^{d}} \omega_{r}(f,(x,\mathbf{y});\delta;M(\mathbb{R}\times \mathbb{R}^{d})) d\bold y dx\\
        &\leq& I_{N}+\epsilon,
    \end{eqnarray*}

where $$I_{N}= \int_{[-\mathbb N, \mathbb N]}\int_{[-\mathbb N, \mathbb N]^{d}} \omega_{r}(f,(x,\mathbf{y});\delta;M(\mathbb{R}\times \mathbb{R}^{d})) d\bold y dx. $$
From Proposition 1, we obtain
$$I_{N}\leq \tau_{r}\left(f;\delta;M\left([-N-1,N+1]\times[-N-1,N+1]^{d}\right)\right)_{1,1}. $$
Using Property (iii), we get
$$\lim_{\delta\to 0}\tau_{1}(f;\delta;M(\mathbb{R}\times \mathbb{R}^{d}))_{1,1} =0\ \implies \lim_{\delta\to 0}\tau_{r}(f;\delta;M(\mathbb{R}\times \mathbb{R}^{d}))_{1,1} =0. $$
Since $f$ is locally Riemann integrable on $\mathbb{R}\times \mathbb{R}^{d}$ hence it is Riemann integrable on $[-N-1,N+1]\times[-N-1,N+1]^{d}$. Therefore from Proposition $2$, we have
$$\lim_{\delta\to 0+}\tau_{r}\left(f;\delta;M\left([-N-1,N+1]\times[-N-1,N+1]^{d}\right)\right)_{1,1} = 0 .$$
This completes the proof.
\end{proof}
\begin{prop}
    Let $\displaystyle f\in M\left([a,b]\times\prod_{i=1}^{d}[c_{i},d_{i}]\right)$, $1\leq q\leq p<\infty$, $r\in \mathbb N$ and
  $$\displaystyle \lim_{\delta\to 0+} \frac{\tau_{r}\left(f;\delta;M\left([a,b]\times\prod_{i=1}^{d}[c_{i},d_{i}]\right)\right)_{p,q}}{\delta^{r}}=0.$$
  Then $f$ is a polynomial of degree $\leq (d+1)(r-1)$.
\end{prop}
\begin{proof}
    Let us define
    $$ \phi(x):= f(x,y_{1},...,y_{d})\ \mbox{where each} \ y_{i}, i=1,2,...,d \ \mbox{are constant}$$
    similarly, for $j=1,2,...,d$, we define
    $$\phi(y_{j}):= f(x,y_{1},...,y_{d})\ \mbox{where each} \ x and y_{i}, i\neq j \ \mbox{are constant} .$$
  Now we show that each of $\phi(x)$ and $\phi(y_{j})$, $j=1,2,..,d$ are polynomials of degree $\leq r-1$.
   In view of Proposition 19 of \cite{Lpbut}, Lemma 2 and the hypothesis we obtain
   \begin{eqnarray*}
    && \displaystyle \lim_{\delta\to 0+} \frac{\tau_{r}\left(f;\delta;M\left([a,b]\times\prod_{i=1}^{d}[c_{i},d_{i}]\right)\right)_{p,q}}{\delta^{r}}=0\\
    &\implies& \displaystyle \lim_{\delta\to 0+} \frac{\tau_{r}\left(\phi(x);\delta;M\left([a,b]\times\prod_{i=1}^{d}[c_{i},d_{i}]\right)\right)_{p,q}}{\delta^{r}}=0\\
     &\implies& \displaystyle \lim_{\delta\to 0+} \frac{\tau_{r}\left(\phi(x);\delta;M\left([a,b]\times\prod_{i=1}^{d}[c_{i},d_{i}]\right)\right)_{1,1}}{\delta^{r}}=0\\
     &\implies& \displaystyle \lim_{\delta\to 0+} \frac{\tau_{r}(\phi(x);\delta;M([a,b]))_{1}}{\delta^{r}}=0\\
      &\implies& \phi(x)\ \mbox{is a polynomial of degree}\leq r-1.
   \end{eqnarray*}
Similarly, we can say that each of $\phi(y_{j})$ is a polynomial of degree $\leq r-1$.
Thus, $f$ is a polynomial of degree $\leq$ degree of $\phi(x)+ \sum_{j=1}^{d}$ degree of $\phi(y_{j})\leq (d+1)(r-1).$
\end{proof}

\begin{prop}
    If $f\in \Lambda^{p,q}$, $1\leq p, q< \infty$, $\delta\geq 0$ and $r\in \mathbb N$ and  $$\displaystyle \lim_{\delta\to 0+} \frac{\tau_{r}(f;\delta;M(\mathbb R \times \mathbb R^{d}))_{p,q}}{\delta^{r}}=0,$$
    then $f(x,\bold y)=0$ for all $(x,\bold y)\in \mathbb R \times \mathbb R^d$.
\end{prop}
\begin{proof}
    From Proposition 1, we have
    $$\displaystyle \lim_{n\to\infty}\tau_{r}\left(f;\delta;M\left([-n,n]\times[-n,n]^{d}\right)\right)_{p,q}= \tau_{r}(f;\delta;M(\mathbb{R}\times \mathbb{R}^{d}))_{p,q}.$$
   This implies that $f$ is polynomial in $[-n,n]\times[-n,n]^{d}$ for arbitrary $n$ in view of Proposition 8. Since $f\in L^{p,q}(\mathbb R \times \mathbb R^{d})$, we get $f(x,\bold y)=0$, for all $(x,\bold y)\in [-n,n]\times[-n,n]^{d}$. Since $n$ is arbitrary, we obtain $f(x,\bold y)=0$ for all $(x,\bold y)\in \mathbb R \times \mathbb R^d$.
\end{proof}
\section{Error Estimation And An Interpolation Theorem}\label{section4}

We define the modified Steklov function for $f\in \Lambda^{p,q}$ as follows
$$ f_{r,h}=-\sum_{j=1}^{r}(-1)^{j}{\binom{r}{j}}F_{r,h_{j}}$$
where $h_{j}=\frac{jh}{r}$, $h>0$, $r\in \mathbb N$ and the usual Steklov function
$$F_{r,h}(x,\bold y)= \frac{1}{h^{r(d+1)}}\int_{[0,h]}...\int_{[0,h]}\int_{[0,h]^{d}}...\int_{[0,h]^{d}} f\left(x+\sum_{i=1}^{r}u_{i},\bold y+\sum_{i=1}^{r}\bold v_{i}\right)d\bold v_{1}...d\bold v_{r}\ du_{1}...du_{r}, $$
here $u_{i}\in \mathbb R$ and $\bold v_{i}\in \mathbb R^{d}$ for $i=1,2,...,r.$
\begin{prop}
     If $f\in \Lambda^{p,q}$, $1\leq p\leq q< \infty$, $r\in \mathbb N$ and $h>0$, then we have $F_{r,h}\in \Lambda^{p,q}$.
\end{prop}
\begin{proof}
    By using Jensen's inequality, we have \\

    \noindent  $ |F_{r,h}(x,\bold y)|^{q}$
    \begin{eqnarray*}
        &\leq&\frac{1}{h^{r(d+1)}}\int_{[0,h]}...\int_{[0,h]}\int_{[0,h]^{d}}...\int_{[0,h]^{d}} \left|f\left(x+\sum_{i=1}^{r}u_{i},\bold y+\sum_{i=1}^{r}\bold v_{i}\right)\right|^{q}d\bold v_{1}...d\bold v_{r} \ du_{1}...du_{r}\\
        &\leq& \frac{2^{q}}{h^{r(d+1)}}\int_{[0,h]}...\int_{[0,h]}\int_{[0,h]^{d}}...\int_{[0,h]^{d}} \left|f\left(x+\sum_{i=1}^{r}u_{i},\bold y+\sum_{i=1}^{r}\bold v_{i}\right)-f(x,\bold y)\right|^{q}d\bold v_{1}...d\bold v_{r} \ du_{1}...du_{r}\\
        &+& \frac{2^{q}}{h^{r(d+1)}}\int_{[0,h]}...\int_{[0,h]}\int_{[0,h]^{d}}...\int_{[0,h]^{d}} |f(x,\bold y)|^{q}d\bold v_{1}...d\bold v_{r} \ du_{1}...du_{r}\\
        &:=& I_{1}+I_{2}.
    \end{eqnarray*}
    Now, for any admissible sequence $(x_{k},\bold y_{\bold j})_{k\in \mathbb Z,\  \bold j\in \mathbb Z^{d}}$ we obtain
    \begin{eqnarray*}
        I_{1}&=& \frac{2^{q}}{h^{r(d+1)}}\int_{[0,h]}...\int_{[0,h]}\int_{[0,h]^{d}}...\int_{[0,h]^{d}} |\Delta_{\sum_{i=1}^{r}u_{i}, \sum_{i=1}^{r}\bold v_{i}}^{1}f(x,\bold y)|^{q}d\bold v_{1}...d\bold v_{r} \ du_{1}...du_{r}\\
        &\leq& \frac{2^{q}}{h^{r(d+1)}}\int_{[0,h]}...\int_{[0,h]}\int_{[0,h]^{d}}...\int_{[0,h]^{d}}[\omega_{1}(f,(x,\bold y),2(r+1)h,M(\mathbb R \times \mathbb R^{d})]^{q}d\bold v_{1}...d\bold v_{r} \ du_{1}...du_{r}\\
         &\leq& 2^{q}\sup_{(x,\bold y)\in Q_{\bold{j}k}}[\omega_{1}(f,(x,\bold y),2(r+1)h,M(\mathbb R \times \mathbb R^{d})]^{q}.
    \end{eqnarray*}
   Now for $I_{2}$, we have
   $$I_{2}\leq 2^{q}\sup_{(x,\bold y)\in Q_{\bold{j}k}}|f(x,\bold y)|^{q}.$$
   Combining $I_{1}\ \mbox{and}\ I_{2}$, we obtain
   \begin{eqnarray*}
       \sup_{(x,\bold y)\in Q_{\bold{j}k}}|F_{r,h}(x,\bold y)|^{q}&\leq &
       2^{q}\sup_{(x,\bold y)\in Q_{\bold{j}k}}[\omega_{1}(f,(x,\bold y),2(r+1)h,M(\mathbb R \times \mathbb R^{d})]^{q} \\
       &+& 2^{q}\sup_{(x,\bold y)\in Q_{\bold{j}k}}|f(x,\bold y)|^{q}
   \end{eqnarray*}
   Now, taking summation on both sides and using triangle inequality, we get
   \begin{eqnarray*}
      \left( \sum_{\bold j \in \mathbb Z^{d}} \sup_{(x,\bold y)\in Q_{\bold{j}k}}|F_{r,h}(x,\bold y)|^{q}\right)^{\frac{p}{q}}&\leq& 2^{p} \left( \sum_{\bold j \in \mathbb Z^{d}} \sup_{(x,\bold y)\in Q_{\bold{j}k}}[\omega_{1}(f,(x,\bold y),2(r+1)h,M(\mathbb R \times \mathbb R^{d})]^{q}\right)^{\frac{p}{q}}\\
       &+& 2^{p}\left( \sum_{\bold j \in \mathbb Z^{d}} \sup_{(x,\bold y)\in Q_{\bold{j}k}}|f(x,\bold y)|^{q}\right)^{\frac{p}{q}}.
   \end{eqnarray*}
   Again, summing over all $k\in \mathbb Z$ on both sides and in view of Proposition 6, we obtain $F_{r,h}\in \Lambda^{p,q}$.
\end{proof}
We obtain the following corollary as an immediate consequence of the foregoing proposition.
\begin{corollary}
    If $f\in \Lambda^{p,q}$, $1\leq p\leq q< \infty$, $r\in \mathbb N$ and $h>0$, then $f_{r,h}\in \Lambda^{p,q}$.
\end{corollary}
In order to Prove Lemma 4, we recall the following Theorem from \cite{kgv}.
\begin{thm}
    For each $\beta \in \mathbb N^{d}$ there are natural numbers $N, b_{k}\in \mathbb N$ and vectors $u^{k}, v^{k}\in \mathbb R^{d}$, $u^{k}=(u_{1}^{k},...,u_{d}^{k}), v_{k}=(v_{1}^{k},...,v_{d}^{k})$, $k=1,2,...,N$ such that
    $$ v_{i}^{k}, v_{i}^{k}+|\beta|u_{i}^{k}\leq |\beta_{i}|, i=1,...,d$$
    and for all $h\in X^{d}$ we have
    $$\Delta^{\beta}(h)=\sum_{v=1}^{N}b_{k}\Delta^{|\beta|}\left(\sum_{i=1}^{d} u_{i}^{k}h_{i}\right) T\left( \sum_{i=1}^{d}v_{i}^{k}h_{i}\right),$$
    where $T:Y^{X}\rightarrow Y^{X}$ is a translation operator defined by $T(h)f(x)=f(x+h)$ and $X,Y$ are linear spaces.

\end{thm}
\begin{lemma}
       If $f\in L^{p,q}(\mathbb R\times \mathbb R^{d})$, $1\leq p\leq q< \infty$, $r\in \mathbb N$ and $h>0$, then there exists a function $f_{r,h}$ with the following estimates:

       \begin{itemize}
           \item[(a)] $ \| f-f_{r,h}\|_{p,q}\leq K(r,d) \ \omega_{r}(f;\delta; L^{p,q}(\mathbb R\times \mathbb R^{d})),$\\

           \item[(b)] $ \delta^{r}\| D^{\beta}f_{r,h}\|_{p,q}\leq K(r,d)\ \omega_{r}(f;\delta; L^{p,q}(\mathbb R\times \mathbb R^{d})), \ |\beta|=r,$\\

           \item[(c)] $ \delta^{|\beta|}\| D^{\beta}f_{r,h}\|_{p,q}\leq K(r,d) \ \omega_{r}(f;\delta; L^{p,q}(\mathbb R\times \mathbb R^{d})),$
       \end{itemize}
       where $K(r,d)$ is positive constant depending only on $r$ and $d$.
\end{lemma}
\begin{proof}
    Let us take
    \begin{eqnarray*}
         f_{r,h}(x,\bold y)&=&- \frac{1}{h^{r(d+1)}}\int_{[0,h]}...\int_{[0,h]}\int_{[0,h]^{d}}...\int_{[0,h]^{d}}\sum_{j=1}^{r}(-1)^{j}{\binom{r}{j}}\\
         &\times& f(x+\frac {j}{r}\sum_{i=1}^{r}u_{i},\bold y+\frac {j}{r}\sum_{i=1}^{r}\bold v_{i})
         d\bold v_{1}...d\bold v_{r} \ du_{1}...du_{r}.
    \end{eqnarray*}
   Then, we have
    \begin{eqnarray*}
        \|f-f_{r,h}\|_{p,q}&\leq& \frac{1}{h^{r(d+1)}}\int_{[0,h]}...\int_{[0,h]}\int_{[0,h]^{d}}...\int_{[0,h]^{d}}\left\| \Delta_{\frac{\sum_{i=1}^{r}u_{i}}{r}, \frac{\sum_{i=1}^{r}\bold v_{i}}{r}}^{1}f(x,\bold y)\right\|_{p,q}\\
        &\times& d\bold v_{1}...d\bold v_{r} \ du_{1}...du_{r}.
    \end{eqnarray*}
  By using the Property 5 of the  local modulus of smoothness, we get
  \begin{equation*}
       \|f-f_{r,h}\|_{p,q}\leq \omega_{r}(f;\sqrt{d}\delta; L^{p,q}(\mathbb R\times \mathbb R^{d}))\leq  K(r,d) \omega_{r}(f;\delta; L^{p,q}(\mathbb R\times \mathbb R^{d})),
  \end{equation*}
  where $K(r,d)$ is positive constant depending only on $r$ and $d$.
  To prove the remaining two estimates, we represent the  usual forward difference in terms of mixed differences as follows:
  $$\Delta_{h,\bold h}^{\beta}=\Delta_{h_{0}}^{n_{0}}\Delta_{\bold h}^{|N|}=\Delta_{h_{0}}^{n_{0}}\Delta_{h_{1}}^{n_{1}}...\Delta_{h_{d}}^{n_{d}},$$

  where $ h_{0}=(h,\bold 0)$, $h_{i}=(0,\bold h_{i})$, $\bold {0}=(0,...,0)\in \mathbb R^{d}$  \mbox{and}  $\bold h_{i}=(0,...,h,...,0)\in \mathbb R^{d}$, $h \mbox{is in the}\ i^{th} \mbox {position }$.
We know that
$$ f_{r,h}=-\sum_{j=1}^{r}(-1)^{j}{\binom{r}{j}}F_{r,h_{j}},$$
where $F_{r,h}$ is the usual Steklov function. Now, we claim that
$$ D^{\beta}F_{r,h_{j}}(x,\bold y)=h^{-|\beta|} \Delta_{\frac{jh}{r},\frac{j\bold h}{r}}^{\beta}f(x,\bold y).$$
The proof of this claim follows from the representation of the finite difference as mentioned above and Proposition 21 (i). The representation allows us to use recursively one-dimensional estimate $(d+1)$ times in different variables to get our required estimate. Thus, we obtain
$$  D^{\beta}f_{r,h}(x,\bold y)=-h^{-|\beta|}\sum_{j=1}^{r}(-1)^{j}{\binom{r}{j}} \Delta_{\frac{jh}{r},\frac{j\bold h}{r}}^{\beta}f(x,\bold y). $$
Applying Theorem 1, we have our desired estimate.
\end{proof}

\begin{prop}
    If $f\in \Lambda^{p,q}$, $1\leq p\leq q< \infty$, $r\in \mathbb N$ and $h>0$, then there exists a function $f_{r,h}\in \Lambda^{p,q}$ with the following properties:
    \begin{itemize}
        \item[(a)] $f_{r,h}\in W^{r}(L^{p,q}(\mathbb R\times \mathbb R^{d})\cap C(\mathbb R\times \mathbb R^{d}))$ and for its $r^{th}$ partial derivative $D^{\beta}f_{r,h}$ with $|\beta|=r$ we have
        $$\| D^{\beta}f_{r,h}\|_{p,q}\leq K(r,d) h^{-r}\tau_{r}(f;\delta; M(\mathbb R\times \mathbb R^{d}))_{p,q}, $$
        \item[(b)] $|f(x,\bold y)-f_{r,h}(x,\bold y)|\leq \omega_{r}(f;(x,\bold y);2(r+1)h;M(\mathbb R\times \mathbb R^{d})),$
        \item[(c)] $ \| f-f_{r,h}\|_{p,q}\leq K(r,d) \tau_{r}(f;2(r+1)h; M(\mathbb R\times \mathbb R^{d}))_{p,q}.$
    \end{itemize}
\end{prop}
\begin{proof}
Applying the Lemma 4 and Proposition 4, we immediately obtain the part(a). Using a similar analysis discussed in Lemma 4, we obtain (b). On taking $L_{p,q}(\mathbb R \times \mathbb R^{d})$ norm on both sides in (b), we obtain (c).
\end{proof}
\begin{prop}
    Let $f\in \Lambda^{p,q}$, $1\leq q\leq p< \infty$, $r\in \mathbb N$ and $\Sigma$ be an admissible partition with upper mesh size $\overline{\Delta}$. Then we have
    $$ \|\omega_{r}(f;.;2h;M(\mathbb R\times \mathbb R^{d}))\|_{l^{p,q}(\Sigma)}\leq c(r)\tau_{r}\left(f;h+\frac{\overline{\Delta}}{r}; M(\mathbb R \times \mathbb R^{d}\right)_{p,q},$$
    where $c(r)$ is a constant depending on $r$.
\end{prop}
\begin{proof}
    Let $\Sigma=(x_{k},y_{j})_{{k\in \mathbb Z},\ {j\in \mathbb Z^{d}}}$ be an admissible sequence with upper mesh size $\overline{\Delta}$. Now fix $k\in \mathbb Z, \ \bold{j}\in \mathbb Z^{d}$ and $z=(x,\bold y), \  \bold y=(y_{1},...,y_{d}).$
    Again let $\xi=(u,\eta)$, $\eta=(\eta_{1},...,\eta_{d})\in \mathbb R^{d}$ and $u\in \mathbb R$ be such that $z,\xi \in Q_{\bold{j}k }$.
    Now we have
    $$ [x-rh,x+rh]\times \prod_{i=1}^{d}[y_{i}-rh, y_{i}+rh] \subseteq [u-rh-\overline{\Delta},u+rh+\overline{\Delta}]\times \prod_{i=1}^{d}[\eta_{i}-rh-\overline{\Delta}, \eta_{i}+rh+\overline{\Delta}].$$
    and,
    \begin{eqnarray*}
      \omega_{r}(f;z;2h;M(\mathbb R\times \mathbb R^{d}))  \leq  \omega_{r}\left(f;\xi;2\left(h+\frac{\overline{\Delta}}{r}\right);M(\mathbb R\times \mathbb R^{d})\right), \  \mbox{for} \ z,\xi \in Q_{\bold{j}k}.
    \end{eqnarray*}
Then, we have
$$\sup_{z\in Q_{\bold{j}k}}\left(\omega_{r}\left(f;z;2h;M(\mathbb R\times \mathbb R^{d}\right)\right)^{q}\Delta_{\bold{j}k}\leq
\int_{Q_{\bold{j}k}}\left(\omega_{r}\left(f;\xi;2\left(h+\frac{\overline{\Delta}}{r}\right);M(\mathbb R\times \mathbb R^{d})\right)\right)^{q}d\xi,$$
which implies that\\
\noindent
$\displaystyle \sum_{\bold j \in \mathbb Z^{d}}\sup_{z\in Q_{\bold{j}k}}\left(\omega_{r}(f;z;2h;M(\mathbb R\times \mathbb R^{d}))\right)^{q}\Delta_{\bold{j}k}$
\begin{eqnarray*}
    \leq   \sum_{\bold j \in \mathbb Z^{d}} \int_{Q_{\bold{j}k}}\left(\omega_{r}\left(f;\xi;2\left(h+\frac{\overline{\Delta}}{r}\right);M(\mathbb R\times \mathbb R^{d})\right)\right)^{q} d\xi.
  \end{eqnarray*}
  Then we have\\
  \noindent
  $\displaystyle \sum_{k\in \mathbb Z}\left(\sum_{\bold j \in \mathbb Z^{d}}\sup_{z\in Q_{\bold{j}k}}\left(\omega_{r}(f;z;2h;M(\mathbb R\times \mathbb R^{d}))\right)^{q}\Delta_{\bold{j}k}\right)^{\frac{p}{q}} $
  \begin{eqnarray*}
 \leq  \sum_{k\in \mathbb Z}\left(\sum_{\bold j \in \mathbb Z^{d}} \int_{Q_{\bold{j}k}}[\omega_{r}(f;\xi;2(h+\frac{\overline{\Delta}}{r});M(\mathbb R\times \mathbb R^{d}))]^{q} d\xi\right)^{\frac{p}{q}}
  \end{eqnarray*}
Therefore we have
  \begin{eqnarray*}
   \|\omega_{r}(f;.;2h;M(\mathbb R\times \mathbb R^{d}))\|_{l^{p,q}(\Sigma)}\leq c(r)\tau_{r}\left(f;h+\frac{\overline{\Delta}}{r}; M(\mathbb R \times \mathbb R^{d})\right)_{p,q}.
\end{eqnarray*}
This completes the proof of the proposition.
\end{proof}
Now, we will discuss the basic interpolation theorem for the $L^{p,q}$ averaged modulus of smoothness in $M(\mathbb R \times \mathbb R^{d})$. Let us consider
$\{\mathscr{L}_{\eta}:\Lambda^{p,q}\longrightarrow L^{p,q}({\mathbb R \times \mathbb R^{d}})\}_{\eta \in \mathbb I}$ be a family of linear operators, $\mathbb I$ being an index set. Further, let $\{\Sigma_{\eta}\}_{\eta\in \mathbb I}$, is a family of admissible partitions $((x_{k,\eta},y_{j,\eta})_{{k\in \mathbb Z},{j\in \mathbb Z^{d}}})$ with upper and lower mesh size
$ \overline{\Delta_{\eta}},\ \underline{\Delta_{\eta}}$ respectively and $ \Delta_{\bold {j}k,\eta}$ as defined previously. We have the following theorem
\begin{thm}
   Let $\{\Sigma_{\eta}\}_{\eta\in \mathbb I}$, be a family of admissible partitions with upper mesh size  $\overline{\Delta_{\eta}}$. Let $\{\mathscr{L}_{\eta}\}_{\eta\in \mathbb I}$ be a family of linear operators from $\Lambda^{p,q}$ into $ L^{p,q}({\mathbb R \times \mathbb R^{d}}), \ 1\leq q\leq p< \infty$ with the following properties:
  \begin{itemize}
      \item[(i)] $\|\mathscr{L}_{\eta}f \|_{L^{p,q}({\mathbb R \times \mathbb R^{d}})} \leq C_{1}\ \|f\|_{l^{p,q}(\Sigma_{\eta})}, $
      \item[(ii)] for $g \in W^{r}(L^{p,q}({\mathbb R \times \mathbb R^{d}}))\cap C(\mathbb R \times \mathbb R^{d})$, we have \\
      $\displaystyle \|\mathscr{L}_{\eta}g- g \|_{L^{p,q}({\mathbb R \times \mathbb R^{d}})}\leq C_{2}\  \overline{\Delta_{\eta}}^{s}\sum_{|\beta|=r}\|D^{\beta}g\|_{L^{p,q}({\mathbb R \times \mathbb R^{d}})},$
  \end{itemize}
  for some fixed $r,s\in \mathbb N$ with $s\leq r $ and $C_{1},\ C_{2}$ are constants. Then for each $f\in \Lambda^{p,q}$ and $ \overline{\Delta_{\eta}}\leq r$ then the following estimate holds
  $$  \|\mathscr{L}_{\eta}f- f \|_{L^{p,q}({\mathbb R \times \mathbb R^{d}})}\leq K(r,d) \tau_{r}\left(f;\overline{\Delta_{\eta}}^{\frac{s}{r}};M(\mathbb R \times \mathbb R^{d})\right)_{p,q}.$$
\end{thm}

\begin{proof}
    Let us take the modified Steklov function $f_{r,h}$ and consider the following
    \begin{eqnarray}
        && \|\mathscr{L}_{\eta}f- f \|_{L^{p,q}({\mathbb R \times \mathbb R^{d}})}\nonumber \\
        &\leq& \|\mathscr{L}_{\eta}(f-f_{r,h}) \|_{L^{p,q}({\mathbb R \times \mathbb R^{d}})}+\|\mathscr{L}_{\eta}f_{r,h}- f_{r,h} \|_{L^{p,q}({\mathbb R \times \mathbb R^{d}})}+\|f_{r,h}- f \|_{L^{p,q}({\mathbb R \times \mathbb R^{d}})}.
    \end{eqnarray}
   From Proposition 12(iii), we get
  \begin{eqnarray*}
      \| f-f_{r,h}\|_{p,q}\leq K(r,d) \tau_{r}(f;2h; M(\mathbb R\times \mathbb R^{d}))_{p,q}.
  \end{eqnarray*}
  Since $\displaystyle \|\mathscr{L}_{\eta}g- g \|_{L^{p,q}({\mathbb R \times \mathbb R^{d}})}\leq C_{2}\  \overline{\Delta_{\eta}}^{s}\sum_{|\beta|=r}\|D^{\beta}g\|_{L^{p,q}({\mathbb R \times \mathbb R^{d}})},$
   and by Proposition 12 (i), we obtain
   \begin{eqnarray*}
   \|\mathscr{L}_{\eta}f_{r,h}- f_{r,h} \|_{L^{p,q}({\mathbb R \times \mathbb R^{d}})}
       &\leq& C(r,d)\overline{\Delta_{\eta}}^{s} \sum_{|\beta|=r}\|D^{\beta}f_{r,h}\|_{L^{p,q}({\mathbb R \times \mathbb R^{d}})} \\
       &\leq& C^{\prime}(r,d) \overline{\Delta_{\eta}}^{s}h^{-r}\tau_{r}(f;h;M(\mathbb R\times \mathbb R^{d}))_{p,q},
   \end{eqnarray*}
   where $C^{\prime}(r,d)= K(r,d) C(r,d)$
  In view of Proposition 13 and Proposition 12 (ii), and by assumption (i), we obtain.
  \begin{eqnarray*}
       \|\mathscr{L}_{\eta}(f-f_{r,h}) \|_{L^{p,q}({\mathbb R \times \mathbb R^{d}})}&\leq&
     C \| f-f_{r,h}\|_{l^{p,q}(\Sigma_{\eta})}\nonumber \\
     &\leq& C \|\omega_{r}(f,.;2h;M(\mathbb R \times \mathbb R^{d})\|_{l^{p,q}(\Sigma_{\eta})} \\
     &\leq& C \tau_{r}(f;(h+\overline{\Delta_{\eta}});M(\mathbb R \times \mathbb R^{d}))_{p,q}.
  \end{eqnarray*}
 Putting $h=\overline{\Delta_{\eta}}^{\frac{s}{r}}$ and in (4.1) and combining the above estimates, we get\\

 \noindent
 $\|\mathscr{L}_{\eta}f- f \|_{L^{p,q}({\mathbb R \times \mathbb R^{d}})}$
 \begin{eqnarray*}
   &\leq&  C \left\{ \tau_{r}(f;\overline{\Delta_{\eta}}^{\frac{s}{r}};M(\mathbb R \times \mathbb R^{d}))_{p,q}+\tau_{r}(f;2\overline{\Delta_{\eta}}^{\frac{s}{r}};M(\mathbb R \times \mathbb R^{d}))_{p,q} \right\} \\
   &\leq& K(r,d)\  \tau_{r}\left(f;\overline{\Delta_{\eta}}^{\frac{s}{r}};M(\mathbb R \times \mathbb R^{d})\right)_{p,q},
 \end{eqnarray*}
 where  $K(r,d)$ is positive constant depending only on $r$ and $d$. This completes the proof of our theorem.

\end{proof}

\section{Rate of Convergence of Neural Network Sampling Operators in Mixed Lebesgue Spaces}\label{section5}
In this section, we establish the rate of approximation for Neural Network Sampling Operators in mixed Lebesgue spaces. Before discussing the approximation results, we shall define the algebraic and absolute moment as follows:

Let $\psi: \mathbb{R}^{2}\rightarrow \mathbb{R}$. Then for any $ \eta \in \mathbb{N}_{0}= \mathbb{N} \cup \{0 \},$ $ n= (n_{1},n_{2})\in \mathbb{N}_{0}^2$ with $ \vert n \vert = n_{1}+n_{2}= \eta,$ we define the algebraic moments of order $\eta$ as
\begin{eqnarray*}
m_{( n_{1},n_{2})} (\psi,u,v):=\sum_{k=-\infty}^{\infty}\sum_{j=-\infty}^{\infty}
\psi(u-k ,v-j)(k-u)^{n_{1}}(j-v)^{n_{2}}
\end{eqnarray*}
and the absolute moments are defined by
\begin{eqnarray*}
M_{( n_{1},n_{2})} (\psi):=\sum_{k=-\infty}^{\infty}\sum_{j=-\infty}^{\infty}
|\psi(u-k ,v-j)|\ |(k-u)|^{n_{1}} |(j-v)|^{n_{2}}.
\end{eqnarray*}

We define
$ \displaystyle M_{\eta}(\psi):= \max_{\vert p \vert = \eta} \  M_{(n_{1},n_{2})} (\psi).$ We note that $ \displaystyle M_{\eta}(\psi)<\infty,$ for every $0\leq \eta\leq \beta$ in view of section 6 of (See[\cite{CSP}]). We can easily see that for $ \xi,\eta \in \mathbb{N}_{0}$ with $\xi < \eta,$ $ M_{\eta}(\chi)<\infty$ implies that $M_{\xi}(\chi)<\infty.$

Let $\psi(x,y)$ be a sigmoidal function and let $I=[-1,1]$. For $f:I\times I\to \mathbb{R},$  the Neural Network Sampling Operators are defined by
\[  F_{n}f(x,y)=  \displaystyle\frac{ \displaystyle\sum_{k=-n}^{n}\sum_{j=-n}^{n}f{\left( \frac{k}{n}, \frac{j}{n}\right)}\psi_{\sigma}(nx-k, ny-j)}{ \displaystyle\sum_{k=-n}^{n}\sum_{j=-n}^{n}\psi_{\sigma}(nx-k, ny-j)}.\]

Now we shall recall some of the crucial properties enjoyed by $\displaystyle \psi_{\sigma}(x,y)$ (See [\cite{DCS}]).

\begin{lemma}\label{n1}
    \begin{itemize}
    \item[]
        \item[(i)] For every $u,v\in \mathbb{R}$, we have $\displaystyle \sum_{k=-\infty}^{\infty}\sum_{j=-\infty}^{\infty}\psi(u-k ,v-j)=1.$
        \item[(ii)] The series $\displaystyle \sum_{k=-\infty}^{\infty}\sum_{j=-\infty}^{\infty}\psi(u-k ,v-j)$ converges uniformly on the compact subsets of $\mathbb{R}\times \mathbb{R}.$
        \item[(iii)] For $u,v\in I$ and $n\in \mathbb{N}$, we have
        $$ \frac{1}{\displaystyle \sum_{k=-\infty}^{\infty}\sum_{j=-\infty}^{\infty}\psi(nu-k ,nv-j)}\leq \frac{1}{\psi_{\sigma}(1,1)}.$$
    \end{itemize}

\end{lemma}

In view of the above Lemma, it is easy to see that the above NN operators $F_{n}$ are well-defined. Indeed, we have
\begin{eqnarray*}
     |F_{n}f(x,y)| \leq \displaystyle\frac{ \displaystyle\sum_{k=-n}^{n}\sum_{j=-n}^{n}\left|f{\left( \frac{k}{n}, \frac{j}{n}\right)}\right|\psi_{\bar{\sigma}}(nx-k, ny-j)}{ \displaystyle\sum_{k=-n}^{n}\sum_{j=-n}^{n}\psi_{\bar{\sigma}}(nx-k, ny-j)}\leq \frac{\|f\|_{\infty}}{\psi_{\sigma}(1,1)}M_{(0,0)}(\psi_{\sigma})<+\infty.
\end{eqnarray*}

\begin{lemma}
    Let $\psi(x,y)$ be a multivariate sigmoidal function defined earlier, $f:I\times I\to \mathbb{R}$ be a bounded measurable function then,
    \[ \|F_{n}f\|_{p,q}\leq \|f\|_{\ell^{p,q}(\Sigma)} (\psi_{\bar{\sigma}}(1,1))^{-\frac{1}{2}(\frac{1}{p}+\frac{1}{q})},     \]
    where $\Sigma=\Sigma_{x}\times \Sigma{y}$ is an admissible partition such that $\Sigma_{x}=\left\{\dfrac{k}{n}: k=-n,...,n\right\}$ and $\Sigma_{y}=\left\{\dfrac{j}{n}: j=-n,...,n\right\}$.
\end{lemma}

\begin{proof}
    By the definition of the $L^{p,q}$ norm and using Jensen's inequality multiple times, we obtain
    \begin{eqnarray*}
        \|F_{n}f\|_{p,q}^{p}&=& \int_{-1}^{1} \left( \int_{-1}^{1} \left| \frac{\sum_{k=-n}^{n}\sum_{j=-n}^{n}f{( \frac{k}{n}, \frac{j}{n})}\psi_{\sigma}(nx-k, ny-j)}{\sum_{k=-n}^{n}\sum_{j=-n}^{n}\psi_{\sigma}(nx-k, ny-j)} \right|^{q} dy\right)^{\frac{p}{q}}dx\\
        &\leq& \frac{1}{(\phi_{\sigma}(1)^{\frac{p}{q}}} \frac{\int_{-1}^{1} \left(\sum_{k=-n}^{n}\sum_{j=-n}^{n}\left| f{( \frac{k}{n}, \frac{j}{n})} \right|^{q} \left( \int_{-1}^{1} \phi_{\sigma}(ny-j) dy\right)\phi_{\sigma}(nx-k) \right)^{\frac{p}{q}} dx}{\sum_{k=-n}^{n} \phi_{\sigma}(nx-k)}\\
        &\leq& \frac{1}{(\phi_{\sigma}(1)^{\frac{p}{q}+1}} \left(\sum_{k=-n}^{n}\left(\sum_{j=-n}^{n}\left| f{\left( \frac{k}{n}, \frac{j}{n}\right)} \right|^{q} \frac{1}{n} \right)^{\frac{p}{q}}\frac{1}{n} \right)\\
        &\leq& \frac{\|f\|_{\ell^{p,q}(\Sigma)}^{p}}{(\psi_{\bar{\sigma}}(1,1))^{-\frac{1}{2}(\frac{p}{q}+1)}}.
    \end{eqnarray*}
    By taking $p^{th}$ square root on both sides of the above expression, we get the desired estimate.
\end{proof}

Now, we will estimate the approximation error in terms of the $\tau$-modulus of smoothness in $L^{p,q}$ norm.

\begin{thm}
    Let $\sigma$ be a sigmoidal function satisfying $(3)$ with $\beta\geq 2\max{(p,q)}$, $1\leq q\leq p < \infty$. Then for any $f\in M(I\times I)$, we have \[  \|F_{n}f-f\|_{p,q} \leq C_{1} \tau_{1}\left( f, \frac{1}{n}\right)_{p,q}+ C_{2} \tau_{2}\left( f, \frac{1}{n}\right)_{p,q},      \] where $n\in \mathbb{N}$ and $C_{1}, C_{2}$ are positive constants.
\end{thm}

\begin{proof}
    Let $0<h\leq1$. Putting $r=2$ in Proposition $12$ we have,
    $$\|F_{n}f-f\|_{p,q}\leq \|F_{n}f-F_{n}f_{2,h}\|_{p,q}+\|F_{n}f_{2,h}-f_{2,h}\|_{p,q}+\|f_{2,h}-f\|_{p,q}. $$
    Now using Lemma $5$, we obtain
    \begin{eqnarray*}
        \|F_{n}f-F_{n}f_{2,h}\|_{p,q}&\leq &\|F_{n}(f-f_{2,h})\|_{p,q}\\
        &\leq & \|f-f_{2,h}\|_{\ell^{p,q}(\Sigma)}(\psi_{\bar{\sigma}}(1,1))^{-\frac{1}{2}(\frac{1}{p}+\frac{1}{q})}\\
        &\leq & (\psi_{\bar{\sigma}}(1,1))^{-\frac{1}{2}(\frac{1}{p}+\frac{1}{q})} \left( \displaystyle \sum_{k\in \mathbb Z}\left(\sum_{\bold j \in \mathbb {Z}}\left(\omega_{2}\left(f;\left(\frac{k}{n}, \frac{j}{n}\right);2h\right)\right)^{q}\frac{1}{n}\right)^{\frac{p}{q}}\frac{1}{n}\right)^{\frac{1}{p}}\\
        &\leq & \frac{2^{\frac{1}{q}+6}}{(\psi_{\bar{\sigma}}(1,1))^{-\frac{1}{2}(\frac{1}{p}+\frac{1}{q})}}\ \tau_{2}\left(f, h+\frac{1}{2n} \right)_{p,q}.
    \end{eqnarray*}
    Now, in view of Proposition $12 \ (iii)$, we get
    \[ \|f_{2,h}-f\|_{p,q}\leq K(r)\ \tau_{r}(f,2h)_{p,q} .  \]
    Further, $f_{2,h}\in W_{p,q}^{2}$. Using local Taylor formula for $f$ at $(x,y)$, there exists a bounded function $\zeta$ such that $\lim_{(t,v)\to (0,0)}\zeta(t,v)=0$ and we have
    \begin{eqnarray*}
        f_{2,h}\left(\frac{k}{n}, \frac{j}{n}\right)=&& f_{2,h}(x,y)+\frac{\partial f_{2,h}}{\partial x}(x,y)\left( \frac{j}{n}-y\right)+ \frac{\partial f_{2,h}}{\partial y}(x,y)\left( \frac{j}{n}-y\right)\\
        &&+ \ \frac{1}{2} \frac{\partial^2 f_{2,h}}{\partial x^2}(x,y) \left( \frac{k}{n}-x\right)^{2} + \ \frac{1}{2} \frac{\partial^2 f_{2,h}}{\partial y^2}(x,y) \left( \frac{j}{n}-y\right)^{2}\\
        &&+ \ \frac{1}{2} \frac{\partial f_{2,h}}{\partial x \partial y}(x,y) \left( \frac{k}{n}-x\right)\left( \frac{j}{n}-y\right)\\
        &&+ \ \zeta\left(\frac{k}{n}-x, \frac{j}{n}-y \right)\left[\left( \frac{k}{n}-x\right)^{2}+ \left( \frac{j}{n}-y\right)^{2} \right].
    \end{eqnarray*}
    From the definition of the NN operator, we have
    \begin{eqnarray*}
        F_{n}f_{2,h}(x,y)= &&f_{2,h}(x,y)+ \frac{\frac{\partial f_{2,h}}{\partial x}(x,y)}{n} \frac{\sum_{k=-n}^{n}\sum_{j=-n}^{n} (k-nx)\psi_{\bar{\sigma}}(nx-k, ny-j)}{\sum_{k=-n}^{n}\sum_{j=-n}^{n}\psi_{\bar{\sigma}}(nx-k, ny-j)}\\
        &&+ \frac{\frac{\partial f_{2,h}}{\partial y}(x,y)}{n} \frac{\sum_{k=-n}^{n}\sum_{j=-n}^{n} (j-ny)\psi_{\bar{\sigma}}(nx-k, ny-j)}{\sum_{k=-n}^{n}\sum_{j=-n}^{n}\ \psi_{\bar{\sigma}}(nx-k, ny-j)}\\
        &&+ \frac{\frac{\partial^{2} f_{2,h}}{\partial x^{2}}(x,y)}{n^2} \frac{\sum_{k=-n}^{n}\sum_{j=-n}^{n} (k-nx)^{2}\psi_{\bar{\sigma}}(nx-k, ny-j)}{\sum_{k=-n}^{n}\sum_{j=-n}^{n}\psi_{\bar{\sigma}}(nx-k, ny-j)}\\
        &&+ \frac{\frac{\partial^{2} f_{2,h}}{\partial y^{2}}(x,y)}{n^2} \frac{\sum_{k=-n}^{n}\sum_{j=-n}^{n} (j-ny)^{2}\psi_{\bar{\sigma}}(nx-k, ny-j)}{\sum_{k=-n}^{n}\sum_{j=-n}^{n}\psi_{\bar{\sigma}}(nx-k, ny-j)}\\
        &&+ \frac{\frac{\partial^{2} f_{2,h}}{\partial x \partial y}(x,y)}{n^2} \frac{\sum_{k=-n}^{n}\sum_{j=-n}^{n} (k-nx) (j-ny)\psi_{\bar{\sigma}}(nx-k, ny-j)}{\sum_{k=-n}^{n}\sum_{j=-n}^{n}\psi_{\bar{\sigma}}(nx-k, ny-j)}\\
        &&+ \frac{1}{n^2} \frac{\sum_{k=-n}^{n}\sum_{j=-n}^{n}\zeta\left(\frac{k}{n}-x, \frac{j}{n}-y \right)\left[\left( \frac{k}{n}-x\right)^{2}+ \left( \frac{j}{n}-y\right)^{2} \right]\psi_{\bar{\sigma}}(nx-k, ny-j)}{\sum_{k=-n}^{n}\sum_{j=-n}^{n}\psi_{\bar{\sigma}}(nx-k, ny-j)}\\
       &&=f(x,y)+ A_{1} +A_{2} +A_{3} +A_{4} +A_{5} + R.
    \end{eqnarray*}
Using Minkowski inequality in the above expression, we get
 $$\|F_{n}f_{2,h}-f_{2,h}\|_{p,q}\leq \sum_{i=1}^{5}\|A_{i}\|_{p,q} + \|R\|_{p,q}.$$
 One can easily find the following estimates by using Jensen's inequality to each $A_{i}, \ i= 1,...,5$ with Theorem $12 \ (iii)$.
 \[  \|A_{1}\|_{p,q}\leq \frac{M_{1,0}(\psi_{\bar{\sigma}})}{\psi_{\bar{\sigma}}(1,1)} \frac{1}{n} \left\|\frac{\partial f_{2,h}}{\partial x} \right\|_{p,q}\leq C_{1}^{\prime} \frac{1}{n} h^{-1}\tau_{1}(f,h). \]
 \[  \|A_{2}\|_{p,q}\leq \frac{M_{0,1}(\psi_{\bar{\sigma}})}{\psi_{\bar{\sigma}}(1,1)} \frac{1}{n} \left\|\frac{\partial f_{2,h}}{\partial y} \right\|_{p,q}\leq C_{2}^{\prime} \frac{1}{n} h^{-1}\tau_{1}(f,h)  .\]
 \[  \|A_{3}\|_{p,q}\leq \frac{M_{2,0}(\psi_{\bar{\sigma}})}{\psi_{\bar{\sigma}}(1,1)} \frac{1}{n^2} \left\| \frac{\partial^{2} f_{2,h}}{\partial x^{2}} \right\|_{p,q}\leq C_{3}^{\prime} \frac{1}{n^2} h^{-2}\tau_{2}(f,h)  .\]
 \[  \|A_{4}\|_{p,q}\leq \frac{M_{0,2}(\psi_{\bar{\sigma}})}{\psi_{\bar{\sigma}}(1,1)} \frac{1}{n^2} \left\|\frac{\partial^{2} f_{2,h}}{\partial y^{2}} \right\|_{p,q}\leq C_{4}^{\prime} \frac{1}{n^2} h^{-2}\tau_{2}(f,h)   .\]
 \[  \|A_{5}\|_{p,q}\leq \frac{M_{1,1}(\psi_{\bar{\sigma}})}{\psi_{\bar{\sigma}}(1,1)} \frac{1}{n^2} \left\|\frac{\partial^{2} f}{\partial x \partial y}\right \|_{p,q} \leq C_{5}^{\prime} \frac{1}{n^2} h^{-2}\tau_{2}(f,h)   .\]

 Now, we will estimate the norm for the remainder as follows:
 \begin{eqnarray*}
     \|R\|_{p,q}&=& \left \|\frac{1}{n^2} \frac{\sum_{k=-n}^{n}\sum_{j=-n}^{n}\zeta\left(\frac{k}{n}-x, \frac{j}{n}-y \right)\left[\left( \frac{k}{n}-x\right)^{2}+ \left( \frac{j}{n}-y\right)^{2} \right]\psi_{\bar{\sigma}}(nx-k, ny-j)}{\sum_{k=-n}^{n}\sum_{j=-n}^{n}\psi_{\bar{\sigma}}(nx-k, ny-j)}\right\|_{p,q}\\
     &\leq& D_{1}+D_{2}.
 \end{eqnarray*}
where, $D_{1}$ is estimated by
 \begin{eqnarray*}
     D_{1}&=&  \left \|\frac{1}{n^2} \frac{\sum_{k=-n}^{n}\sum_{j=-n}^{n}\zeta\left(\frac{k}{n}-x, \frac{j}{n}-y \right)\left( \frac{k}{n}-x\right)^{2} \psi_{\bar{\sigma}}(nx-k, ny-j)}{\sum_{k=-n}^{n}\sum_{j=-n}^{n}\psi_{\bar{\sigma}}(nx-k, ny-j)}\right\|_{p,q}\\
     &\leq& \|\zeta\|_{\infty} \frac{1}{n^2} \left \| \frac{\sum_{k=-n}^{n}\sum_{j=-n}^{n}\left( \frac{k}{n}-x\right)^{2} \psi_{\bar{\sigma}}(nx-k, ny-j)}{\sum_{k=-n}^{n}\sum_{j=-n}^{n}\psi_{\bar{\sigma}}(nx-k, ny-j)}\right\|_{p,q}\\
     &\leq&   \|\zeta\|_{\infty} \frac{1}{n^2} \frac{M_{2,0}(\psi_{\bar{\sigma}})}{\psi_{\bar{\sigma}}(1,1)}.
 \end{eqnarray*}
 Similarly, we obtain
$$ D_{2}\leq  \|\zeta\|_{\infty} \frac{1}{n^2} \frac{M_{
0,2}(\psi_{\bar{\sigma}})}{\psi_{\bar{\sigma}}(1,1)}.$$

Since $D_{1}$ and $ D_{2} $ are independent of $f$ and $\zeta$, so for sufficiently large $n$, we obtain that $D_{1},D_{2}\to 0$ as $n\to \infty$. This implies that $R \to 0$ as $n \to \infty$. Hence, by combining all the estimates, we get the desired estimate.
\end{proof}

\section{Examples and Applications}\label{section6}

In this section, we discuss some well-known sigmoidal functions which will satisfy our conditions. As a first example, we consider the logistics sigmoidal function. The logistics sigmoidal function is defined by
${\sigma_{\ell}}(u) := \dfrac{1}{1 + e^{-u}}, \quad u\in \mathbb{R}.$ It is well known that ${\sigma_{\ell}}(u)$ satisfies the conditions $(A_1)-(A_3)$ (see \cite{FCZ}). Since $\psi_{\sigma_{\ell}}(u,v):= \sigma_{\ell}(u).\sigma_{\ell}(v),$ so $\psi_{\sigma_{\ell}}(u,v)$ also satisfies conditions $(A_1)-(A_3)$.\\
Next, we consider the hyperbolic tangent sigmoidal functions as follows:
\[ \psi_{\sigma_{H}}(u,v) := \frac{1}{2}(\tanh{u}+1) \frac{1}{2}(\tanh{v}+1),\]
where $$\tanh{z}=\frac{e^{2z}-1}{e^{2z}-1}.$$ By (see \cite{DCS}), $\tanh{z}$ satisfies conditions $(A_1)-(A_3)$ and hence $\psi_{\sigma_{H}}(u,v)$ also satisfies conditions $(A_1)-(A_3)$.

\subsection{Implementation Results}

Now, we consider the function continuous function  $$f(x,y)=x^2 \sin(x+y)+y^2\cos(xy)$$ and its graph is shown in figure \ref{fig:plot1}.
\begin{figure}[h]
			\centering
			\includegraphics[width=0.5\textwidth]{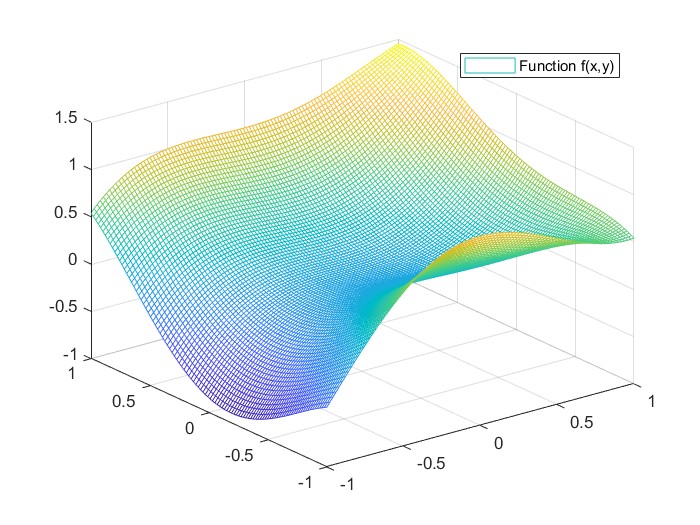}
			\caption{The function $f$   }
			\label{fig:plot1}
		\end{figure}

The accuracy of the approximation $ \|F_{n}f - f\|_{p,q}$ using the logistic activation function $\sigma_{l}(x)\sigma_{l}(y)$ and hyperbolic tangent function $\frac{1}{4}(\tanh (x +1))(\tanh (y +1))$ are presented in Table \ref{tab:error} and Figures \ref{fig:plot2} and \ref{fig:plot3}

\begin{figure}[h]
			\centering
			\includegraphics[width=0.5\textwidth]{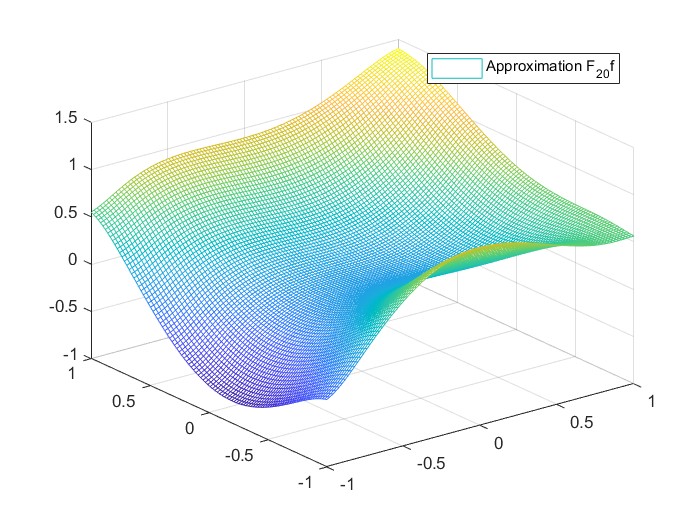} 
			\caption{Approximation  $F_{20}f $ of $f$   using sigmoidal function $\sigma_{l}$.}
			\label{fig:plot2}
		\end{figure}
\begin{figure}[h]
			\centering
			\includegraphics[width=0.5\textwidth]{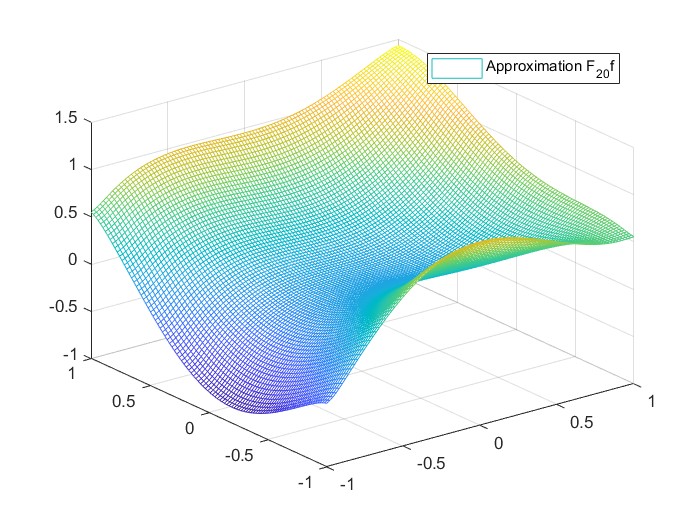}
			\caption{Approximation  $F_{20}f $ of $f$   using hyperbolic tangent function .}
			\label{fig:plot3}
		\end{figure}
		\begin{table}[h]
			\centering
			\begin{tabular}{|l|c|c|c|c|}
				\hline
				$n$ & Activation function & $(p,q)$&$\| F_{n}f - f\|_{p,q} $ &$\| F_{n}g - g\|_{p,q} $   \\
				\hline
				20 &$\sigma_{l}(x)\sigma_{l}(y)$& (2,3) &0.0642 &0.8129\\
                \hline
                30&$\sigma_{l}(x)\sigma_{l}(y)$&(2,3)&0.0385&0.7062\\
                	\hline
               20& $\frac{1}{4}(\tanh (x +1))(\tanh (y +1))$  &(2,3) &0.0263 &0.7638\\
                		\hline
                30& $\frac{1}{4}(\tanh (x +1))(\tanh (y +1))$  &(2,3) &0.0161&0.6597\\
                	\hline
			\end{tabular}
			\caption{Errors   $\|F_{n}f-f\|_{p,q}$ and   $\|F_{n}g-g\|_{p,q}$ in approximation}  
			\label{tab:error}
		\end{table}
Next, we  consider the discontinuous function $$g(x,y)=\left\{\begin{array}{lll}
x^2\sin(x+y)+y^2\cos(xy)+1 & if & x> 0\\
-x^2\sin(x+y)-y^2\cos(xy)-1 & if & x \le 0.
\end{array}\right.$$
The graph of $g$ is shown in figure \ref{fig:plot4}. The accuracy of the approximation $ \|F_{n}g - g\|_{p,q}$ using the logistic activation function $\sigma_{l}(x)\sigma_{l}(y)$ and hyperbolic tangent function $\frac{1}{4}(\tanh (x +1))(\tanh (y +1))$ are presented in Table \ref{tab:error} and Figures \ref{fig:plot5} and \ref{fig:plot6}
\begin{figure}[h]
			\centering
			\includegraphics[width=0.5\textwidth]{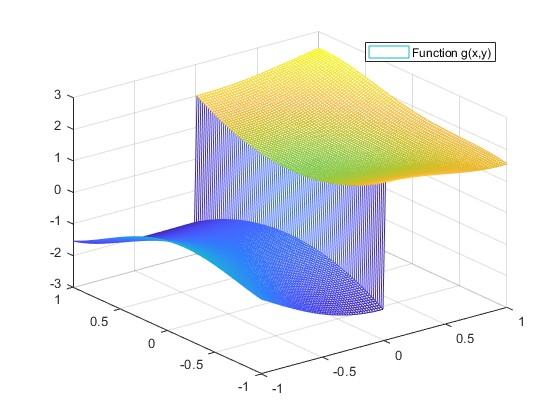} 
			\caption{The function $g$}
			\label{fig:plot4}
		\end{figure}

\begin{figure}[h]
			\centering
			\includegraphics[width=0.5\textwidth]{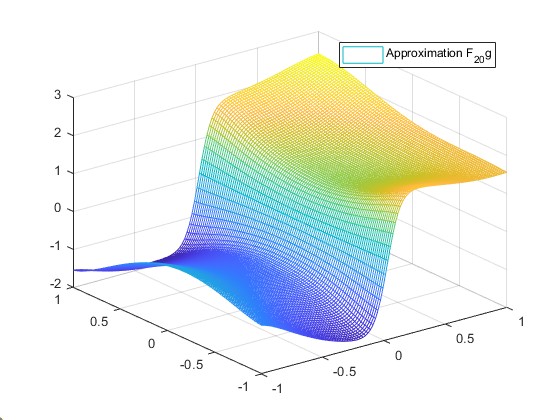} 
			\caption{Approximation  $F_{20}g $ of $g$   using sigmoidal function.}
			\label{fig:plot5}
		\end{figure}

\begin{figure}[h]
			\centering
			\includegraphics[width=0.5\textwidth]{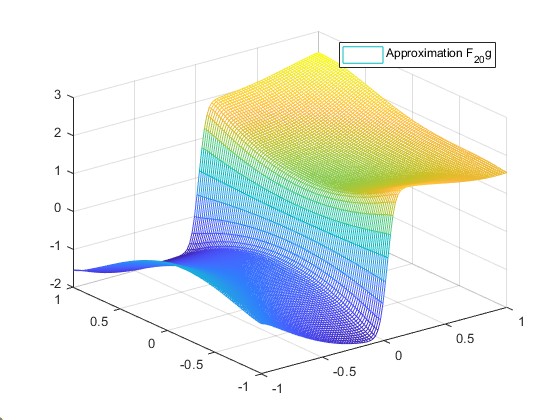} 
			\caption{Approximation  $F_{20}g $ of $g$   using hyperbolic tangent function.}
			\label{fig:plot6}
		\end{figure}

\section{Conclusion}
The rate of approximation of neural network operators in-terms of averaged modulus of smoothness is proved in this paper. In order to achieve our approximation, the properties of averaged modulus of smoothness has been discussed in mixed norm. Further, suitable subspaces $\Lambda^{p,q}(\mathbb R\times \mathbb R^{d}),$ of $L^{p,q}(\mathbb{R}\times\mathbb{R})$ are identified and established that averaged modulus of smoothness is finite for all functions in 
$\Lambda^{p,q}(\mathbb R\times \mathbb R^{d}).$ Using the properties of averaged modulus of smoothness, the rate of approximation of certain linear operators in mixed Lebesgue norm is estimated. As an application of these linear operators, rate of approximation of neural network operators in-terms of averaged modulus of smoothness in mixed norm has been proved. Finally, some examples of sigmoidal functions has been discussed. Using these sigmoidal functions, implementation of continuous and discontinuous functions by neural network operators has been provided.

\end{document}